\documentclass[preprint]{elsarticle}

\usepackage{latexsym}
\usepackage[all]{xy}
\usepackage{amsmath}
\usepackage{amsthm}
\usepackage{amssymb}
\usepackage{amsfonts}
\usepackage{enumerate}
\usepackage[latin1]{inputenc}

\newcommand{\ie}{{\em i.e. }}
\newcommand{\z}{\textrm{-}}
\newcommand{\tto}{\longrightarrow}

\newcommand{\uHom}{\underline{\textrm{Hom}}}

\newcommand{\NN}{\mathbb{N}}

\newcommand{\ZZ}{\mathbb{Z}}
\renewcommand{\AA}{\mathbb{A}}
\newcommand{\GG}{\mathbb{G}}
\newcommand{\colim}{\mathrm{co}\!\lim}

\newcommand{\deux}{{\underline{2}}}

\newtheorem{thm}{Theorem}
\newtheorem{prop}[thm]{Proposition}
\newtheorem{cor}[thm]{Corollary}
\newtheorem{lemma}[thm]{Lemma}
\newtheorem{sublemma}[thm]{Sub-lemma}
\theoremstyle{definition}
\newtheorem{defi}[thm]{Definition}

\begin{document}
\author{Mathieu Anel}
\address{{\sc cirget, uq\`am}, Case postale 8888, Succursale centre-ville\\ Montr\'eal, Qu\'ebec, Canada, H3C 3P8 }
\ead{anel.matthieu@courrier.uqam.ca}
\ead[url]{http://thales.math.uqam.ca/\~{}anelm}
\title{Grothendieck topologies from unique factorisation systems}
\begin{abstract}
This work presents a way to associate a Grothendieck site structure to a (locally presentable) category endowed with a unique factorisation system of its arrows. In particular this recovers the Zariski and Etale topologies and others related to Voevodsky's cd-structures. As unique factorisation systems are also frequent outside algebraic geometry, the same construction applies to some new contexts, where it is related with known structures defined otherwise. The paper details situations in algebraic geometry and sketches only the other examples.
\end{abstract} 
\maketitle

\setcounter{tocdepth}{2}
\tableofcontents

\newpage
\section{Introduction}

This work is about how certain Grothendieck topologies and a theory of spectra can be generated from unique factorisation systems. A particular case of our construction will be the Zariski and Etale topologies and associated spectra of algebraic geometry, and others close to Voevodsky's cd-structures. As unique factorisation systems are also frequent outside algebraic geometry, the same construction (slightly generalized) applies to some new contexts, where it is often related with known structures defined otherwise. The paper details situations in algebraic geometry and sketches only the other examples.
Some of the results are well known but the systematic presentation using unique factorisation system is new.

\paragraph*{Topological interpretation of lifting diagrams}

In a category $\cal C$, a lifting diagram is a commutative diagram as follows
$$\xymatrix{
P\ar[d]_u\ar[r]&U\ar[d]^f\\
N\ar[r]\ar@{-->}[ur]^{\ell}&X.
}$$
The arrow $\ell$, when it exists, is called a {\em lift of $u$ through $f$}.
The diagram is called a {\em lifting diagram} if a lift exist, and a {\em unique lifting diagram} if the lift exists and is unique.
In this last case, $u$ (resp. $f$) is said left (resp. right) orthogonal to $f$ (resp. $u$).
A {\em lifting system} is defined as two classes of maps $\cal A,B\subset C$ such that each map of $\cal A$ is left orthogonal to any of $\cal B$ and such that $\cal A$ and $\cal B$ are saturated for this relation (\S\ref{liftsys}). $\cal A$ and $\cal B$ are called respectively the left and the right classes.

\medskip
We propose the following topological interpretation of a unique lifting diagram: all objects are to be thought as spaces, the composite map $P\to X$ is a point of $X$ (in the generalized sense of 'family of points'), the map $u:P\to N$ is a neighbourhood (or an infinitesimal thickening) of $P$, the map $N\to X$ say that this neighbourhood is "in" $X$, the map $f:U\to X$ is a open immersion $X$, and the map $P\to U$ says that $U$ contains the point $P$.
The unique lifting property then reads:
{\em in a space $X$, any open $U$ containing a point $P$ contains every neighbourhood $N$ of $P$ contained in $X$},
which is exactly the fundamental intuition behind the classical definition of open subsets of topological spaces. We propose here an approach of topology based on this remark.

\paragraph*{Factorisation systems}

A {\em factorisation system} on a category $\cal C$ is the data of two classes of maps $\cal A,B\subset\cal C$ and a factorisation $X\to \phi(u)\to Y$ of any map $u:X\to Y\in{\cal C}$ such that $X\to \phi(u)\in{\cal A}$ and $\phi(u)\to Y\in{\cal B}$. The factorisation is said {\em unique} if $\phi(u)$ is unique up to a unique isomorphism (\S\ref{factosys}). The two classes $\cal (A,B)$ of a unique factorisation system define always a unique lifting system and the converse is true if $\cal C$ is locally presentable and $\cal (A,B)$ is of small generation (proposition \ref{factolim}). The unique lifting systems that will appear in this paper will all be associated to unique factorisation systems.

The topological interpretation of a lifting system is even better if is associated to a factorisation system: 
if $P\to G\to X$ is the factorisation of $p:P\to X$, still thinking $P$ as a point and $X$ as a space, one can interpret $G\to X$ as being the germ at the point and $P\to G$ as the embedding of the point in its germ.
Such a construction of germs does not exist in Topology (where germs need to be pro-objects) but it is well known in algebraic geometry and we are going to explain that it is associated with a factorisation system.

Indeed, the Zariski topology has the particularity that Zariski open embeddings between affine schemes are all in the right class of a unique factorisation system $(Cons^o,Loc^o)$ on $CRings^o$ (\S\ref{zariski}). $Loc$ is the class of localisations of rings and $Cons$ the class of conservative maps of rings: a map $u:A\to B$ is conservative if $u(a)$ invertible implies $a$ invertible, an example is the map $A\to k$ from a local ring to its residue field.
Any map of rings $u:A\to B$ factors in a localisation followed by a conservative map $A\to A[S^{-1}]\to B$ where $S=u^{-1}(B^\times)$.
In particular this factorisation applied to a map $u:A\to k$ where $k$ is a residue field of $A$ gives $A\to A_p\to k$ where $A_p$ is the local ring of $A$ at the kernel of $p$ of $u$.
Geometrically, $A\to k$ corresponds to a point $p$ of $X=Spec_{Zar}(A)$, $N=Spec_{Zar}(A_p)\to X$ is the germ of $X$ at $p$ and $P=Spec_{Zar}(k)\to N$ is the embedding of a point into some neighbourhood, which is coherent with our above  interpretation. 
Also, if $U\to X$ is a Zariski open subset of $X$ containing $P$, this data define a lifting square as before and the existence of the lift $N\to U$ is a consequence of $N$ being the limit of all $U\to X$ containing $P$.

\medskip
With the previous considerations in mind, it is tempting to look at a unique factorisation system $\cal (A,B)$, the following way: the right class $\cal B$ would be formed of open embeddings and the left class $\cal A$ of infinitesimal neighbourhoods.
But the example of Zariski topology, show us also that only finitely presented  localisations of rings are to be thought as open embeddings, so a general map in the right class should rather be thought as a "pro"-open embedding.
Also, it is possible to see using a topological intuition, that a map lifting uniquely the neighbourhood of some point, once given a lift of the point is not in general an open embedding but rather a local homeomorphism (an etale map). The  other example of the Etale topology (\S\ref{etale}) makes it very clear.

So finally, we are going to propose an interpretation of the class $\cal B$ of a unique lifting system as a class of "pro"-etale maps.
A unique lifting systems is then though as a theory of pro-etale maps and a tool to develop abstract analogs of Zariski and Etale topologies.

\medskip
We list here the four unique factorisation systems on the category $CRings$ of commutative rings that we are going to study in the sequel .
\begin{center}
\begin{tabular}{c|c|c}
Name & Left class & Right class\\
\hline
Zariski &localisations & conservative maps\\
Etale & ind-etale maps & henselian maps\\
Domain & surjections & monomorphisms\\
Finite & ind-finite maps & integrally closed maps
\end{tabular}
\end{center}
A factorisation system $\cal (A,B)$ on $\cal C$ defines another one $({\cal A}^o,{\cal B}^o)$ on ${\cal C}^o$ and we will in fact have more interests on the opposite systems of the previous four. To each of them will be associated a Grothendieck topology on the opposite category $CRings^o$ of commutative rings, the third one being Voevodsky's plain lower cd-topology in \cite{voevodsky} restricted to affine schemes.

\paragraph*{Results}
From a factorisation system in the opposite of a locally presentable category, we built a general scheme associating to it:
\begin{itemize}
\item a notion of {\em etale map} (\S\ref{etalemap}),
\item a notion of {\em points} of an object (\S\ref{points}),
\item a notion of {\em local objects} (\S\ref{localobject}),
\item a Grothendieck topology (called the {\em factorisation topology}) which covering families are etale families surjective on points (\S\ref{pointcovers}),
\item two {\em toposes} functorialy associated to any object $X$ and called the {\em small and big spectra} of $X$, the big one being always a retraction of the big one (\S\ref{spectra}),
\item and a {\em structural sheaf} on the small spectra of $X$ whose stalks are the "local forms" of $X$, \ie pro-etale local objects over $X$ (\S\ref{structuresheaf})
\end{itemize}
Then the main result of the paper is theorem \ref{theorem} allowing one to compute the categories of global points of the spectra using the local objects. We refer to it as the moduli interpretation of the spectra, but we won't study fully the moduli aspects of our spectra in this paper, such a study would require a much more topossic approach than we have chosen here and will be the subject of another paper \cite{anel1}.

In the case of the four systems on $CRings^o$ these notions give:
\begin{center}
{\scriptsize
\begin{tabular}{p{2cm}|p{2cm}|p{2cm}|p{2cm}|p{2cm}}
&Zariski &Etale & Domain &Finite\\
\hline
Etale maps& Zariski open maps & etale maps & Zariski closed embeddings & finite maps\\
Points&nilpotent extension of fields & nilpotent extension of separably closed fields & fields & algebraically closed fields\\
Local objects&local rings & strict henselian local rings & integral domains & strict integrally closed domains (\S\ref{strictintegrallycloseddomain})\\
Small spectrum of $A$ & usual Zariski spectrum (topos classifying all localisations of $A$) &usual Etale spectrum (topos classifying all strict henselisation of $A$) & a topos classifying all quotients domains of $A$ & a topos classifying strict integral closure of quotient domains of $A$\\
Big spectrum of $A$& usual big Zariski topos classifying local $A$-algebras & usual big Etale topos classifying strict henselian local $A$-algebras & a topos classifying $A$-algebras that are integral domains & a topos classifying $A$-algebras that are strict integrally closed domains.\\
\end{tabular}
}
\end{center}

\paragraph*{Nisnevich contexts}

Nisnevich topology on schemes is defined by etale covering families satisfying a lifting property for maps from spectra of fields. Such a lifting property cannot in general be obtained by a single etale map and this does not distinguish a class of maps that could be part of a factorisation system. For this reason Nisnevich topology is not a factorisation topology, but it defines an interesting operation on such that we called {\em Nisnevich forcing}. It consists to force a class of objects to be local objects (\S\ref{forcing}) by selecting only the covering families of the factorisation topology that lift maps from the objects of the forcing class. Applied to the Etale topology and the class of fields, this gives the usual Nisnevich topology. But an interesting other case is to apply this, still with the class of field as forcing class, to the Finite topology (\S\ref{nisfinite}) the resulting topology is then the lower cd-structure of Voevodsky in \cite{voevodsky} restricted to affine schemes.

The data of a factorisation system and a Nisnevich forcing class is called a {\em Nisnevich context} (def. \ref{niscontext}) and the construction of our spectra (\S\ref{spectra}) as well as theorem \ref{theorem} are defined directly in such a context. The previous table can then be completed by the following one:
\begin{center}
\begin{tabular}{c|p{3cm}|p{3cm}}
&Nisnevich & Nisnevich Finite \\
\hline
Local objects& henselian local rings & integrally closed domains\\
Small spectrum of $A$ & topos classifying ind-etale henselian local $A$-algebra & topos classifying ind-finite integraly closed $A$-algebra \\
Big spectrum of $A$& topos classifying  henselian local $A$-algebra & topos classifying integraly closed $A$-algebra \\
\end{tabular}
\end{center}

\paragraph*{Duality} There seems to be a kind of duality between the Zariski/Etale settings and Domain/Finite settings, \S\ref{duality} regroups some naive elements around this idea.

\paragraph*{Other examples}

Many examples of unique factorisation system exists outside of algebraic geometry and we think a generalisation of our approach can be interesting. We try to motivate this idea by sketching some examples in \S\ref{other}.

The first example is in the setting of To\"en and Vaqui\'e algebraic geometry under $spec(\ZZ)$ \cite{fun}, but it will be developped fully in another work \cite{anelvaquie}.
The next two examples deal with the $(Epi,Mono)$ factorisation systems that always exist in a topos or an abelian category, the notion of point corresponds to irreducible objects and the associated spectra are essentielly discrete spaces.
Another example study the factorisation systems on the category of small categories given by initial (resp. final) functors and discrete left (resp. right) fibrations. The associated spectra of a category $C$ are respectively the toposes of covariant and of contravariant functors. Moreover this example share a duality of the same flavour of that of etale and finite maps.
We study also a factorisation system on the category of simplicial sets left generated by inclusion of faces of simplices, points and local objects are vertices and a more interesting situation is obtained forcing all simplices to be local object. For this topology, the small spectrum of a simplicial set is related to the cellular dual of the usual geometric realisation (where vertices correspond to open cells).

\paragraph*{Plan of the paper}

Section \ref{facto} consists in some recollections and lemmas about lifting and factorisation systems, the main result is theorem \ref{factolim} describing the construction of a unique factorisation from a lifting system of small generation. It will be used in \S\ref{etale} to construct a factorisation system related to the Etale topology. All this can be skipped on first reading.

Section \ref{topo} is the core of the article. It develops the topological interpretations and constructions associated to a factorisation system. It uses the strong hypothesis that the category ${\cal C}^o$ is the opposite of a finitely presentable category, this restriction is motivated by the example of commutative rings and more generally by categories of algebras over a Lawvere theory (it will be applied to monoids in \cite{anelvaquie}).
The notion of a Nisnevich context and the associated small and big spectra are defined in \S\ref{forcing} and \S\ref{spectra}.
The theorem of computation of their points is in \S\ref{moduliinterpretation} and their expected structure is proven in \S\ref{spectra} and \S\ref{structuresheaf}.

Section \ref{examples} develops the examples: six in algebraic geometry and others outside.

Finally, an appendix compares our work to some other on the subject of spectra.

\section*{Motivations and acknowledgments}

The origin of this work was to understand why Zariski and Etale topologies where coming with both notions of small and big topos and a class of distinguished maps playing the role of "open embeddings", the classical theory of Grothendieck (pre)topologies being insufficient to explain this extra structure.
It is Andr\'e Joyal that suggested to me that Zariski topology should be related to the $(Loc,Cons)$ factorisation system on commutative rings. Although he won't be satisfied with the way I've chosen to present the ideas here, this paper have been influenced by numerous conversations with him. 
It is after a conversation with Georges Maltsiniotis, that I had the idea for the notion points, I am particularly grateful to him for listening the first one my raw ideas and for his remarks.
I learn first the excellent philosophy of thinking the spectrum of an object $X$ as the moduli spaces of some "local forms" of $X$ from Joseph Tapia, although all this is not fully described here this have been influential, i'm grateful to him for our conversations on the subject.
I'm also grateful to Jonathan Pridham for pointing out to me that the orthogonal class of henselian maps should be that of ind-etale ones and to Fr\'ed\'eric D\'eglise for a remark on finite maps.

Most of this study has been worked out during the excellent 2007-2008 program on Homotopy Theory and Higher Categories in Barcelona's {\sc crm}, I'm very grateful to the organizers for inviting me and funding me all year. It has been written while I was staying at Montreal's {\sc cirget} that I thank also for invitation and funding.

\paragraph*{Notations}

For an object $X$ of a category $\cal C$ the category of objects of $\cal C$ under $X$ is noted $X/{\cal C}$ and that of objects over $X$ ${\cal C}/X$.
For a category $\cal C$ its category of arrows is noted ${\cal C}^{\underline{2}}$.
given two maps $A\to B$ and $A\to X$ in a category $\cal C$, their pushout is written $X\to X\cup_AB$.
$CRings$ is the category of commutative rings. $\cal S$ will denote the topos of sets.


\section{Lifting properties and factorisation systems}\label{facto}

We recall the notion of lifting and factorisation systems from \cite{bousfield,joyal}.

\subsection{Lifting systems}\label{liftsys}

In a commutative diagram square
$$\xymatrix{
P\ar[d]_u\ar[r]&U\ar[d]^f\\
N\ar[r]\ar@{-->}[ur]^\ell&X.
}$$
the map $u$ is said to have the {\em unique left lifting property} with respect to $f$ and 
the map $f$ is said to have the {\em unique right lifting property} with respect to $u$ if it exist a unique diagonal arrow $\ell$ making the two obvious triangles commutative. The arrow $\ell$ will be called the lift or the lifting.

\medskip

Let ${\cal B}$ a class of maps of ${\cal C}$, a map $u:X\to Y\in {\cal C}$ is said to be {\em left (resp. right) orthogonal} to ${\cal B}$ iff it has the unique left (resp. right) lifting property with respect to all maps of ${\cal B}$.
The class of maps left (resp. right) orthogonal to ${\cal B}$ is noted $^\bot\!{\cal B}$ (resp. ${\cal B}^\bot$).
If $\cal A\subset B$ then ${\cal B}^\bot\subset {\cal A}^\bot$ and $^\bot\!{\cal B}\subset ^\bot\!{\cal A}$.

\begin{defi}
The data of two classes $\cal A,B$ of maps of $\cal C$ such that ${\cal A}= ^\bot\!{\cal B}$ and ${\cal B}= {\cal A}^\bot$ is called a {\em unique lifting system} on $\cal C$. Such a system is noted $\cal (A,B)$.
\end{defi}

\medskip
For a class $G$ of maps of $\cal C$ we define ${\cal B}=G^\bot $ and ${\cal A}=^\bot\!{\cal B}$.
\begin{lemma}\label{leftgen}\label{setgen}
The previous classes $\cal A$ and $\cal B$ form a unique lifting system.
\end{lemma}
\begin{proof}
We must show ${\cal B}={\cal A}^\bot$. By construction $G\subset {\cal A}$ so ${\cal A}^\bot\subset G^\bot={\cal B}$, and the inverse inclusion is a consequence of ${\cal A}=^\bot\!{\cal B}$.
\end{proof}

Such a factorisation system will be qualified as {\em left generated by the set $G$}. There is a dual notion of right generation.

\medskip

A class $\cal B$ of maps in a category $\cal C$ has the {\em left cancellation property} if for any $X\overset{u}{\tto} Y\overset{v}{\tto} Z$ such that $vu$ and $v$ are in $\cal B$, so is $u$. When $\cal B$ is stable by composition, this is equivalent to say that for any object $X$, ${\cal B}/X$ is a full subcategory of ${\cal C}/X$. The dual notion is called right cancellation.

Here are some properties of the classes of a lifting system.
\begin{prop}\label{proplift}
\begin{enumerate}
\item ${\cal A}$ and ${\cal B}$ are stable by composition. 
\item ${\cal A}\cap {\cal B}$ is the class of isomorphisms of ${\cal C}$.
\item ${\cal B}$ is stable by pullback and has the left cancellation property. In particular any section or retraction of a map in ${\cal B}$ is in ${\cal B}$. (The dual statement holds for $\cal A$.)
\item In the category of arrows of $\cal C$, any limit of maps in ${\cal B}$ is in ${\cal B}$. (The dual statement holds for $\cal A$.)
\item (Codiagonal property) the class $\cal A$ contains the codiagonals of its morphisms (see proof). (The dual statement holds for $\cal B$.)
\end{enumerate}
\end{prop}
\begin{proof}
The first and second properties are left to the reader.

\textit{3.} Stability by composition and pullback are easy. We are going to prove that for a map $a:X\to Y\in \cal C$, the class $a^\bot$ has the left cancellation property.
Let $u:Z\to T$ and $v:T\to U\in\cal B$ such that $vu\in \cal B$, for any square
$$\xymatrix{
X\ar[d]_a\ar[r]&Z\ar[d]^u\\
Y\ar@{-->}[ru]^\ell\ar[r]_q&T
}$$
we are looking for a lift $\ell$. Composing at the bottom by $v$ gives
$$\xymatrix{
X\ar[d]_a\ar[r]&Z\ar[d]_u\ar@/^1pc/[dd]^{vu}\\
Y\ar@{=}[d]\ar[r]&T\ar[d]_v\\
Y\ar@{-->}[ruu]^s\ar[r]&U
}$$
and a lift $s$ of $a$ through $vu$. We need to show that this is the good one, \ie that $us=q$.
This can be seen in
$$\xymatrix{
X\ar[d]_a\ar[r]&Z\ar[d]_u\ar[r]^u&T\ar[d]^v\\
Y\ar@{-->}[ru]^s\ar[r]_q&T\ar[r]\ar@{=}[ru]&U
}$$
as $us$ and $q$ give two lifts of $a$ through $v$. The conclusion follows as classes having the cancellation property are stable by intersection.

\medskip
\textit{4.}, let $I$ be the interval category $\{0\to 1\}$, ${\cal C}^I$ is the arrow category of $\cal C$, $\cal B$ is a subclass of the class of objects of ${\cal C}^I$. If $D:{\cal D}\to {\cal C}$ is a diagram of arrows all in $\cal B$, then, if the limit of this diagram exists, it is in $\cal B$. Indeed, let $Z_d\to T_d\in\cal B$ be the value of the diagram $D$ at $d$ and $Z\to T$ be the limit of $D$, the existence of a lift $\ell$ for a square
$$\xymatrix{
X\ar[d]_{a\in{\cal A}}\ar[r]&Z\ar[d]\\
Y\ar@{-->}[ru]^\ell\ar[r]&T
}$$
is equivalent to the existence of lift for all
$$\xymatrix{
X\ar[d]_{a\in{\cal A}}\ar[r]&Z_d\ar[d]\\
Y\ar@{-->}[ru]^{\ell_d}\ar[r]&T_d
}$$
such that for $\delta:d\to d'\in \cal D$
$$\xymatrix{
X\ar[d]_{a\in{\cal A}}\ar[r]&Z_d\ar[d]\ar[r]^{\zeta_\delta}&Z_{d'}\ar[d]\\
Y\ar@{-->}[ru]^{\ell_d}\ar@{-->}[rru]_{\ell_{d'}}\ar[r]&T_d\ar[r]&T_{d'}
}$$
${\zeta_\delta}\circ {\ell_d}=\ell_{d'}$, but this is a consequence of the unicity of the lift.

\medskip
\textit{5.} The codiagonal of a morphism $A\to B$ is the map $B\cup_AB\to B$. It is a retract of the inclusion $B\to B\cup_AB$ which is a pushout of $A\to B$ so it is in $\cal A$ is $A\to B$ is. Then the cancellation property for $\cal A$ ensures $B\cup_AB\to B\in \cal A$ too.
\end{proof}

The following lemma gives an interesting equivalence between the right cancellation property and having codiagonals.
\begin{lemma}\label{diagocancel}
A subcategory $G$ of $\cal C$ stable by cobase change satisfies the right cancellation iff it contains the codiagonals of all its morphisms.
\end{lemma}
\begin{proof}
For $u:A\to B\in G$, $i_1:B\to B\sqcup_AB$ is in $G$ as cobase change of $u$ along itself.
If $G$ has right cancellation, $\delta_u:B\sqcup_AB\to B$ is in $G$ as $\delta_u\circ i_1=id_B$.
For $u:A\to B$ and $v:B\to C$ such that $u,vu\in G$, we want to prove that $v\in G$.
The square
$$\xymatrix{
B\sqcup_AB\ar[d]_{v\sqcup_Aid_B}\ar[r]^-{\delta_v}&B\ar[d]^u\\
C\sqcup_AB\ar[r]^-w&C
}$$
is a pushout. If $G$ is stable by codiagonals $w\in G$. Then as $i_1:C\to C\sqcup_AB$ is in $G$ as a pushout of $vu$, so is $u=w\circ i_1$.
\end{proof}

\medskip
To finish we mention that there is an obvious notion of a general (non unique) lifting system. The following result says that unicity of the lift is a property of a non unique lifting system.

\begin{lemma}\label{uniquelift}
A general lifting system $\cal (A,B)$ is unique iff the $\cal A$ is stable by codiagonals iff the $\cal B$ is stable by diagonals.
\end{lemma}
\begin{proof}
We are going to work only with the condition on $\cal A$.
Suppose we have a square
$$\xymatrix{
A\ar[r]\ar[d]&C\ar[d]\\
B\ar[r]\ar@<.6ex>@{-->}[ru]^{\ell_1}\ar@<-.6ex>@{-->}[ru]_{\ell_2}&D
}$$
with two lifts. These two lifts agree iff the following square have a lift:
$$\xymatrix{
B\cup_AB\ar[r]\ar[d]&C\ar[d]\\
B\ar[r]^{(\ell_1,\ell_2)}\ar@{-->}[ru]&D.
}$$
\end{proof}

\subsection{Factorisation systems}\label{factosys}

\begin{defi}
A {\em unique factorisation system} on a category $\cal C$ is the data of two classes $\cal A,B$ of maps in $\cal C$ such that any arrow $u:X\to Y$ admits a factorisation
$$\xymatrix{
&\phi(u)\ar[rd]^{\beta(u)}\\
X\ar[rr]^u\ar[ur]^{\alpha(u)}&&Y
}$$
with $\alpha(u)\in \cal A$ and $\beta(u)\in \cal B$, which is unique up to unique isomorphism, \ie for two such factorisations
$X\to \phi(u)\to Y$ and $X\to \varphi(u)\to Y$ there exists a unique isomorphism $\phi(u)\to\varphi(u)$ making the two obvious triangles commuting.
\end{defi}

For short, such a factorisation system will be noted $\cal C=(A,B)$. It is obvious that $({\cal B}^o,{\cal A}^o)$ is another factorisation system on ${\cal C}^o$.

\medskip
The definition of a unique factorisation system has many consequences toward the following lemma.
\begin{lemma}\label{fundamental}
In $\cal C=(A,B)$, any commuting square
$$\xymatrix{
X\ar[r]\ar[d]_{\alpha(u)}&Z\ar[d]^b\\
\phi(u)\ar[r]_{\beta(u)}\ar@{-->}[ur]^\ell&Y
}$$
where $X\to \phi(u)\to Y$ is a factorisation of some $u:X\to Y$ and $b\in\cal B$, admits a unique lifting $\ell$.
\end{lemma}
\begin{proof}
This follows by considering a factorisation of $X\to Z$ and using the uniqueness of the decomposition of $u$.
\end{proof}

The dual lemma considering a square
$$\xymatrix{
X\ar[r]\ar[d]_a&\phi(u)\ar[d]\\
Z\ar[r]\ar@{-->}[ur]^\ell&Y
}$$
with $a\in \cal A$ is also true.

\begin{cor}
The classes $\cal A$ and $\cal B$ of a unique factorisation system define a unique lifting system.
\end{cor}
\begin{proof}
Given a commuting square
$$\xymatrix{
X\ar[r]\ar[d]_a&Z\ar[d]^b\\
Y\ar[r]&T
}$$
with $a\in\cal A$ and $b\in\cal B$, the result follows by considering a factorisation of the diagonal $X\to T$ and by the above lemma and its dual.
\end{proof}
 
We are going to see in \S\ref{liftfacto} that the converse is true if $\cal C$ is nice enough.

\begin{prop}
This factorisation is functorial in the sense that, for any commuting square
$$\xymatrix{
X\ar[d]\ar[r]^u&Y\ar[d]\\
X'\ar[r]_v&Y'
}$$
and any choice of factorisation of $u$ and $v$ is associated a unique map $\varphi(u)\to \varphi(v)$ such that the following diagram commutes:
$$\xymatrix{
X\ar[r]^{\alpha(u)}\ar[d]&\varphi(u)\ar[r]^{\beta(u)}\ar[d]&Y\ar[d]\\
X'\ar[r]_{\alpha(v)}&\varphi(v)\ar[r]_{\beta(v)}&Y'.
}$$
\end{prop}
\begin{proof}
From $$\xymatrix{
X\ar[r]^{\alpha(u)}\ar[d]&\varphi(u)\ar[r]^{\beta(u)}&Y\ar[d]\\
X'\ar[r]_{\alpha(v)}&\varphi(v)\ar[r]_{\beta(v)}&Y'.
}$$
one can extract the square
$$\xymatrix{
X\ar[r]^{\alpha(u)}\ar[d]&\varphi(u)\ar[d]\\
\varphi(v)\ar[r]_{\beta(v)}&Y'.
}$$
Then, the wanted map exists by the previous corollary.
\end{proof}

\begin{lemma}\label{relativefacto}
If $\cal C= (A,B)$ is a unique factorisation system, then for any $X\in {\cal C}$, the category $X\setminus {\cal C}$ can be equipped with a unique factorisation system $X\setminus {\cal C}=({\cal A}_X,{\cal B}_X)$ where the factorisation of a map $X\to Y\to Z$ is essentially that of $Y\to Z$ in $\cal C$.
\end{lemma}

By duality, the categories ${\cal C}/X$ also inherit a unique factorisation system.

\subsection{From lifts to factorisations}\label{liftfacto}

We are going to study how to construct a unique factorisation system from a unique lifting system.
For left generated non unique lifting systems, the small object argument is used to construct a non unique factorisation system. This construction works also for unique lifting system and the resulting factorisation can be proved to be unique as a consequence of the stability of the right class by diagonals (lemma \ref{uniquelift}). 
But as unique lifting systems are somehow simpler than non unique ones, one can expect a simpler description of the factorisation in this case. 

In the case were $\cal C$ is a locally presentable category, we provide here such a construction, noticeable for not using the axiom of choice (not as the small object argument). This will be used in \ref{etale}.

\medskip

Let ${\cal C}$ be a cocomplete $\lambda$-accessible category (\ie a locally presentable category) and ${\cal C}^\lambda$ the full subcategory of $\lambda$-presentable objects. ${\cal C}^\lambda$ is stable by $\lambda$-small colimits and ${\cal C}=Ind_\lambda({\cal C}^\lambda)$.
Let $G$ be some set of maps in $\cal C$ and $\cal (A,B)$ the lifting system left generated by $G$.
The existence of left generators authorizes a simple enough description of the class $\cal B$ and we would like to have also some description of the class $\cal A$ from $G$.

We define $G'$ as the smallest class of $({\cal C}^\lambda)^\deux$ containing $G$ and stable by composition, $\lambda$-small colimits, pushouts along maps of ${\cal C}^\lambda$ and right cancellation. As it is stable by composition, $G'$ can be seen as a subcategory of ${\cal C}^\lambda$.
We define also $\overline{G}$ as the class of maps of $\cal C$ that are pushouts of maps in $G'$.
$\overline{G}$ is called the class of {\em strict $\lambda$-presentable maps of the class $\cal A$}
(the class of {\em $\lambda$-presentable maps of the class $\cal A$} is ${\cal A}\cap {\cal C}^\lambda$).

\begin{lemma}\label{gbar}
$\overline{G}$ is the smallest subclass of ${\cal C}^\deux$ containinig $G$ stable by composition, $\lambda$-small colimits 
composition, pushouts along maps of ${\cal C}$ and right cancellation. 
\end{lemma}
\begin{proof}
Let $G''$ be the the smallest subcategory of $\cal C$ containinig $G$ stable by composition, pushouts and right cancellation. By construction $\overline{G}\subset G''$ so it is enough to prove that $\overline{G}$ is stable by composition, $\lambda$-small colimits and right cancellation.

\begin{sublemma}\label{replacement}
Given a pushout square
$$\xymatrix{
A\ar[r]\ar[d]&X\ar[d]\\
B\ar[r]&Y
}$$
where $A\to B\in {\cal C}^\lambda$ and a map $C\to Y$, with $C\in{\cal C}^\lambda$, there exists another
pushout square
$$\xymatrix{
A'\ar[r]\ar[d]&X\ar[d]\\
B'\ar[r]&Y
}$$
where $A'\to B'\in {\cal C}^\lambda$ such that $C\to Y$ factors through $B'\to Y$.
\end{sublemma}
\begin{proof}
Define a category where objects are factorisations $A\to A'\to X$ where $A\to A'\in {\cal C}^\lambda$, 
and maps from $A\to A'\to X$ to $A\to A''\to X$ are maps $A'\to A''$ making the two obvious triangles to commute. 
This category is that of $\lambda$-presentable objects in the accessible category of all factorisations of $A\to X$, thus $X$ is the filtered colimit of the $A'$. Then $Y$ can be presented as the filtered colimit of the $A'\cup_AB$ and the map $C\to Y$ has to factor through one of the $A'\cup_AB$ for some $A'$.
\end{proof}

\noindent {\sl Composition.}
Consider a triangle $X_0\to X_1\to X_2$ where the two $X_i\to X_{i+1}$ for $i=0,1$ are in $\overline{G}$, we want to show that $X_0\to X_2$ is in $\overline{G}$.
$X_i\to X_{i+1}$ is the pushout of some map $A_i\to B_i$ in $G'$ and in general $B_0\not=A_1$ but sub-lemma \ref{replacement} ensures that we can find another $A_0\to B_0$ such that there exist a map $A_1\to B_0$.
Then all we have to do is replace $A_1\to B_1$ by $B_0\to B_0\cup_{A_0}B_1$.

\noindent {\sl Right cancellation.}
A preliminary remark: for some maps $X\to Y_i$ in $\overline{G}$ written as pushouts of $a_i:A_i\to B_i$ along some maps $A_i\to X$, it is possible to replace the $a_i$ by $a'_i:\sqcup_i A_i\to B_i\sqcup (\sqcup_{j\not=i}A_j)$, which are maps with a same source. Those maps are still in $G'$ as pushouts of $a_i$ along $A_i\to \sqcup_j A_j$.
Now consider a triangle $X\to Y_1\to Y_2$ where the two $X\to Y_i$ are in $\overline{G}$, we want to show that $Y_1\to Y_2$ is in $\overline{G}$.
The $X\to Y_i$ can be written as pushouts of maps $A\to B_i$ in $G'$ along some map $A\to X$. 
In general there is no map $B_1\to B_2$ in $G'$ whose pushout along $B_1\to Y_1$ is $Y_1\to Y_2$ but, by sub-lemma \ref{replacement}, it is possible to find some other $A\to B_i$ such that this map exists.
We now have two maps $A\to B_2$ that need not to commute but they are equalized by $B_2\to Y_2$ so, by sub-lemma \ref{replacement} again, we can assume that they do.
Finally, we have written $X\to Y_1\to Y_2$ as the pushout of a triangle $A\to B_1\to B_2$. By right cancellation $B_1\to B_2$ is in $G'$ so $Y_1\to Y_2$ is in $\overline{G}$.

\noindent {\sl $\lambda$-small colimits.}
Let $u_i:X_i\to Y_i$ be a diagram in $\overline{G}\subset{\cal C}^\deux$ indexed by a $\lambda$-small category $I$.
For each $i\in I$, there exists a map $a_i:A_i\to B_i\in G'$ such that $u_i$ is the pushout of $a_i$ along some $A_i\to X_i$. We are going to prove that all these $a_i$ can be chosen to form a $I$ diagram in $G'$. It requires a few steps: first, replace $A_i$ by $A'_i=\sqcup_{j\to i} A_j$ and $a_i$ by its pushout $a'_i:A'_i\to B'_i$ along $A_i\to A'_i$; there is a canonical map $A'_i\to X$ and $u_i$ is a pushout of $a'_i$ along it. For any $u:j\to i$, this create a map $A'_j\to A'_i$ and to obtain a map $B'_j\to B'_i$ we need to apply sub-lemma \ref{replacement}. Let $A^u_i\to B^u_i$ be the replacement of $a'_i$ obtained this way, to supress the dependance on $u$ we put them all together by taking their colimit along $I/i$. 
Let denote $a''_i:A''_i\to B''_i$ this $I$-diagram, as $G'$ is stable by $\lambda$-small colimits, all $a''_i$ are in $G'$ and so is their colimit.
Finally, the colimit of $X_i\to Y_i$ is in $\overline{G}$ as the pushout of the colimit of $A''_i\to B''_i$.
\end{proof}

The idea is now the following: for a lifting system $\cal (A,B)$, suppose we have a factorisation of a morphism $u:X\to Y$ in $X\overset{\alpha}{\tto}X'\overset{\beta}{\tto}Y$ with $\alpha\in\cal A$ and $\beta\in\cal B$, then for any square
$$\xymatrix{
X\ar[r]^\alpha\ar[d]_{\in \cal A}&X'\ar[d]^\beta\\
U\ar[r]\ar@{-->}[ru]^s&Y
}$$
there exists a unique $s$. This suggests to build $X\to X'$ as a colimit of all $X\to U$. 

For this colimit to exist, we are going to consider only those $X\to U$ in $\overline{G}$. Define $G_u$ to be the category whose objects are compositions $X\to U\to Y$, where $X\to U$ is in $\overline{G}$, and whose morphisms are diagrams
$$\xymatrix{
X\ar[r]\ar@{=}[d]&U_1\ar[d]\ar[r]&Y\ar@{=}[d]\\
X\ar[r]&U_2\ar[r]&Y
}$$
(by cancellation the map $U_1\to U_2$ is still in $\overline{G}$). 

\begin{lemma}
$G_u$ is a small with all $\lambda$-small colimits; so in particular it is $\lambda$-filtered.
\end{lemma}
\begin{proof}
$G_u$ is small as all maps $X\to U\in \overline{G}$ are pushouts of maps in $G'$ that is small.
Let $X\to U_i\to Y$ be a diagram indexed by a $\lambda$-small category $I$.
By lemma \ref{gbar}, the colimit $X\to U$ of the $X\to U_i$ exists in $X/\overline{G}$ and $X\to U\to Y$ is a colimit  in $G_u$ for the $X\to U_i\to Y$. 
\end{proof}

\begin{thm}\label{factolim}
Let $\cal C$ be a cocomplete $\lambda$-accessible category and $G$ a set of maps between $\lambda$-presentable objects, left generating a unique lifting system $\cal (A,B)$.
For a map $u:X\to Y$ there exists a factorisation $u:X\to \colim_{G_u}U\to Y$, where $G_u$ is as before,
such that $X\to \colim_{G_u}U$ is in $\cal A$ and $\colim_{G_u}U\to Y$ is in $\cal B$.

This defines a unique factorisation system on $\cal C$ whose left class is $\cal A$ and right class is $\cal B$,
moreover maps in $\cal A$ are $\lambda$-filtered colimits (taken in ${\cal C}^\deux$) of strict $\lambda$-presentable maps (${\cal A}=Ind_\lambda\z\overline{G}$).
\end{thm}
\begin{proof}
Let first remark that the factorisation system has to be unique as $\cal A$ is stable by codiagonals (lemma \ref{uniquelift}).

As $G_u$ is filtered, it is connected so the map $X\to \colim_{G_u}U$ is the colimit of maps $X\to U$ in ${\cal C}^{\underline{2}}$ and thus in $\cal A$.
Then $\colim_{G_u}U\to Y$ is in $\cal B$ if, for any $g:A\to B\in G$, any square
$$\xymatrix{
A\ar[d]_g\ar[r]^-u& \colim_{G_u}U\ar[d]\\
B\ar[r]\ar@{-->}[ru]&Y
}$$
admit a unique diagonal filler.
The map $u$ factors through some $A\to U$, the pushout out $U\to U\cup_AB\in \overline{G}$ define an element $X\to U\cup_AB\to Y$ of $G_u$ and a map $U\cup_AB\to \colim_{G_u}U$ which gives the wanted lift.
The last assertion is clear by construction of the factorisation.
\end{proof}

In the previous theorem, the generating set $G$ can always be replaced by $G'$ and more canonically by the whole ${\cal A}^\lambda={\cal A}\cap {\cal C}^\lambda$.

\begin{cor}\label{indleftgen}
If $\cal C$ is locally presentable and $\cal C=(A,B)$ is a left generated by $G\subset {\cal A}^\lambda$ then, for any $X\in {\cal C}$, 
$$
X\setminus {\cal A} \simeq Ind_\lambda(X\setminus {\cal A}^\lambda)
$$
\end{cor}
\begin{proof}
According to theorem \ref{factolim} any map $X\to U\in{\cal A}$ is a $\lambda$-filtered colimit in ${\cal C}^\deux$ of some maps $X_i\to U_i\in {\cal A}^\lambda$. Defining $X\to U'_i\in X\setminus {\cal A}^\lambda$ as the pushout of $X_i\to U_i$ along $X_i\to X$, the diagram $i\mapsto X\to U'_i$ is still filtered and $X\to U$ is still the colimit of the $X\to U_i$ but this colimit can now be taken in $Ind_\lambda(X\setminus {\cal A}^\lambda)$.
\end{proof}

\begin{lemma}\label{relativegenfacto}
If $\cal C$ is locally presentable and ${\cal (A,B)}$ is a left generated factorisation system, then for any $X\in {\cal C}$, the factorisation system $X\setminus {\cal C}=({\cal A}_X,{\cal B}_X)$ defined in lemma \ref{relativefacto} is a left generated factorisation system.
\end{lemma}
\begin{proof}
If $\cal C$ is locally presentable so is $X\setminus{\cal C}$ and, if $G$ is the set of left generators for ${\cal (A,B)}$, that of left generators for $({\cal A}_X,{\cal B}_X)$ is the set of maps $X\cup g: X\cup A\to X\cup B$ where $g\in G$.
A map $X\to Y\to Z$ in $X\setminus{\cal C}$ is in ${\cal B}_X)$ iff it has the right lifting property with respect to any $g:A\to B\in G$ but such a lifting square can always be factored as :
$$\xymatrix{
A\ar[r]\ar[d]&X\cup A\ar[r]\ar[d]& Y\ar[d]\\
B\ar[r]&X\cup B\ar[r]& Z
}$$
where the left hand square is a pushout. From what a map $X\to Y\to Z$ in $X\setminus{\cal C}$ is in ${\cal B}_X)$ iff it has the right lifting property with respect to any $X\cup g:X\cup A\to X\cup B$.
\end{proof}

\newpage

\section{Topology}\label{topo}

This section presents the topological interpretation of factorisation systems sketched in the introduction.
The construction will take the form of a covariant functor 
\begin{eqnarray*}
Spec:{\cal C}&\tto &{\cal T}opos\\
X&\longmapsto& Spec(X)
\end{eqnarray*}
(where ${\cal T}opos$ is the category of toposes and geometric morphisms up to natural isomorphisms) sending the right class $\cal B$ of the factoristion system to pro-etale maps of toposes.

\subsection{Etale maps}\label{etalemap}

We will assume some properties on $\cal C$ to have a nice theory: ${\cal C}$ will be the opposite of a locally finitely presentable category, that is ${\cal C}=Pro({\cal C}^\omega)$ where ${\cal C}^\omega$ is the subcategory of ${\cal C}$ corresponding to finitely presented objects of ${\cal C}^o$.
A map in $\cal C$ is said to be {\em of finite presentation} if it is a pull-back of some map in ${\cal C}^\omega$
The class of such maps is noted ${\cal C}^f$; the same argument as in lemma \ref{gbar} can be applied to prove that ${\cal C}^f$ is stable by composition, base change along any map of $\cal C$, and has left cancellation.
If $*$ is the terminal object of ${\cal C}$, ${\cal C}^\omega\simeq{\cal C}^f/*$.

\medskip
The initial object of $\cal C$ is assumed to be strict (any map to it is an isomorphism), it is called {\em empty} and noted $\emptyset$.

\begin{defi}
Given a factorisation system ${\cal C=(A,B)}$:
\begin{enumerate}[a.]
\item a map $U\to X$ in $\cal B$ is called {\em p-etale} and a {\em p-etale open} of $X$ (the name is chosen suggest pro-etale);
\item a p-etale map $U\to X$ in called an {\em etale map} and an {\em etale open of $X$} if it is in ${\cal C}^f$. 
\end{enumerate}
\end{defi}

The intersections ${\cal A}\cap {\cal C}^f$ and ${\cal B}\cap {\cal C}^f$ are noted respectively ${\cal A}^f$ and ${\cal B}^f$. 
${\cal B}^f$ is the category of etale maps and ${\cal B}^f{/X}$ the category of etale opens of $X$. By left cancellation of $\cal B$ and ${\cal C}^f$, ${\cal B}^f{/X}$ is a full subcategory of ${\cal C}{/X}$.

\subsection{Points}\label{points}

In algebraic geometry, Zariski covering families are defined as surjective on a given set of points.
The set of points of the spectrum of a ring $A$ can be characterized via some equivalence relation on the set of maps from $A$ to fields. We propose here a notion that will play the role of fields and use it to define the set of points of any object.

\begin{defi}
Given a factorisation system ${\cal C=(A,B)}$:
\begin{enumerate}[a.]
\item an object $P$ of $\cal C$ is called a $\cal (A,B)${\em-point} (or only a {\em point} if the context is clear) if it is not empty and if any map $U\to P\in{\cal B}^f$ where $U$ is non empty has a section (non necessarily unique);
\item a {\em point of an object $X$} is a map $x:P\to X$ from a point $P$;
\item the {\em category of points of an objet $X\in{\cal C}$}, noted ${\cal P}t_{{\cal B}^f}(X)$, is the subcategory of ${\cal C}{/X}$ span by objects $P\to X$, where $P$ is a point; in particular, ${\cal P}t_{{\cal B}^f}({\cal C}):={\cal P}t_{{\cal B}^f}(*)$, is the subcategory of $\cal C$ spanned by all points;
\item the {\em set of points} of an object $X$, noted $pt_{{\cal B}^f}(X)$ is defined as the set of connected components of ${\cal P}t_{{\cal B}^f}(X)$.
\end{enumerate}
\end{defi}


\medskip
In topological terms, points are those objects such that any etale map has a section, if one thinks monomorphic etale maps as open embeddings, a point will have in particular no non trivial opens. This is one argument for the name "point" for this notion. Also, in the study of rings, our points will correspond to various kinds of rings closed under some operations (inverses, algebraic elements...) extracting the classes of fields, separably closed fields... which are indeed the "points" of algebraic geometry.

\subsection{Point covering families}\label{pointcovers}

We can now use the previous notion of point to copy the notion of surjectivity used in algebraic geometry for covering families.

\begin{prop}
Given a family of etale maps $\{U_i\to X\}$, the following two properties are equivalent:
\begin{enumerate}
\item Any point $P\to X$ lift to one of the $U_i$
$$\xymatrix{
&U_i\ar[d]^{\exists i}\\
P\ar[r]\ar@{-->}[ru]&X
}$$
\item The induced map of sets $\sqcup_i pt_{{\cal B}^f}(U_i)\tto pt_{{\cal B}^f}(X)$ is surjective.
\end{enumerate}
\end{prop}
\begin{proof}
It is clear that \textit{1.} implies \textit{2.} Conversely, \textit{2.} says that for any $P\to X$, there exists an $i$ and a morphism $P'\to P$ from another point $P'$ such that $P'\to X$ lift to $U_i$. But this forces $U_i\times_XP$ to be non empty and as $\cal B$ is stable by base change, $U_i\times_XP\to P$ must then have a section.
\end{proof}

\begin{defi}\label{factotopo}
A family $\{U_i\to X\}$ in ${\cal B}^f$ is a {\em point covering family of $X$} if it satisfies one of the above two conditions.
\end{defi}

\begin{prop}
Point covering families of $X$ define a pretopology on ${\cal B}^f{/X}$ and ${\cal C}^f{/X}$.
\end{prop}
\begin{proof}
Our definition of pretopology is taken from \cite[II.1.3.]{SGA4-1}.
Maps in $\cal B$ are stable by pullbacks in $\cal B$ or in $\cal C$, and so are maps surjective on points (easy from the definition): any pullback of a point covering family is again a point covering family. Identities are in ${\cal B}$ and surjective on points.
And finally, for $\{U_i\to X,i\}$ and for $\{V_{ij}\to U_i,j\}$ all covering families, all $V_{ij}\to X$ are etale by composition and $\sqcup V_{ij}\to X$ is still surjective on points.
\end{proof}

The associated topology is called the {\em factorisation topology}.

\subsection{Local objects}\label{localobject}

We defined our covering families such that any point lift through them, but many more objects have this lifting property, this is the idea of a local object. Topologically, these objects correspond to germs. This notion is related to that of point of a topos and has nothing to do with factorisation systems, but, in the particular case of factorisation topologies, it gives back many known classes of objects (such as local rings).

We will define in fact two notions of local objects with respect to a factorisation system.
Some local objects as defined above come with a feature generalises the residue field of a local ring: a map in the class $\cal A$ from a point. We call these local objects {\em pointed}. A priori, not all local objects are pointed nor is the map from a point unique when it exists, however this will be true in all our examples in algebraic geometry.

\medskip
A family $\{U_i\to L\}$ is said to {\em have a section} if there exists an $i$ and a section of $U_i\to L$.
A family $\{U_i\to X\}$ is said to {\em have a section along $L\to X$} (or to {\em lift through} $\{U_i\to X,i\}$) if there exists a section of $U_i\times_XL\to L$ for some $i$.

\begin{lemma}\label{localobjects}
For $L\in\cal C$, the following assertions are equivalent:
\begin{enumerate}
\item $L$ is such that every point covering family $\{U_i\to L\}$ admits a section.
\item $L$ is such that every point covering family $\{U_i\to X\}$ has a section along any $L\to X$.
\end{enumerate}
\end{lemma}

\begin{defi}\label{localob}
Any object $L$ satisfying those conditions wil be called {\em local}.
\end{defi}

\begin{lemma}
If $L$ is local and $L\to L'\in\cal A$, then $L'$ is local.
\end{lemma}
\begin{proof}
For $L\to L'\in {\cal A}$, let $\{U_i\to L'\}$ be a point covering family of $L'$, the pulled-back cover $U'_i\to L$ has a section by assumption on $L$ and this give a square where one can use property of the lifting system $\cal (A,B)$:
$$\xymatrix{
L\ar[d]_{\in \cal A}\ar[r]^{\exists i}&U_i\ar[d]^{\in \cal B}\\
L'\ar@{-->}[ru]\ar@{=}[r]&L'.
}$$
\end{proof}

As it is clear that points are local objects, the previous lemma authorizes the construction of local objects by considering targets of maps $P\to L\in \cal A$ where $P$ is a point.

\begin{defi}\label{pointedlocalob}
A {\em pointed local object} of $\cal C$ is an object $L$ such that there exists a point $P\to L\in{\cal A}$.
\end{defi}


\begin{lemma}
The initial object of $\cal C$ is not a local object.
\end{lemma}
\begin{proof}
$\emptyset$ is strict if any map $X\to \emptyset$ is an isomorphism, so, as points are supposed not initial, the set of points of $\emptyset$ is empty. This prove that the empty family is a point covering family of $\emptyset$ and such a family cannot have a section.
\end{proof}

In examples from algebraic geometry this will prove that the zero ring is never a local object for the factorisation topologies.

\subsection{Nisnevich forcing}\label{forcing}

Both Zariski and Etale topologies are examples of factorisation topologies but this is not the case of Nisnevich topology. However, this topology has a definition from the Etale topology that can be generalized to a general context: starting with a topology given by some covering families, the new topology is defined by considering only those covering families satisfying a lifting condition relatively to maps from a given class $\cal L$ of objects (which is the class of fields in the case of Nisnevich topology). 
This operation forces the class $\cal L$ to be points of the subtopos associated to the topology.
We call such a construction a {\em Nisnevich forcing}, it is an operation on Grothendieck topologies and is completely independent of any factorisation system.

\begin{defi}
Let $\cal C$ be a category with a topology $\tau$ defined via some covering families $U_i\to X$, and $\cal L$ a class of objects of $\cal C$.
\begin{enumerate}[a.]
\item A covering family $U_i\to X$ is said {\em $\cal L$-localising} if for any object $L\in{\cal L}$ and any map $L\to X$ lift through the cover.
\item The {\em $\cal L$-Nisnevich forcing} of $\tau$ (refered to for short as the {\em Nisnevich topology}), noted $\tau_{\cal L}$, is the topology generated by $\cal L$-localising covering families. This topology is coarser than $\tau$.
\end{enumerate}
\end{defi}

The class $\cal L$ is called the {\em forcing class}.
The {\em saturation} of $\cal L$, noted $\overline{\cal L}$, is defined as the subcategory of $\cal C$ of local objects (def. \ref{localob}) for the topology $\tau_{\cal L}$, these objects are called {\em Nisnevich local objects}. $\tau_{\overline{\cal L}}=\tau_{\cal L}$ and $\overline{\cal L}$ is maximal for this property.
If ${\cal L}=\emptyset$ then $\tau_{\cal L}=\tau$ and $\overline{\emptyset}={\cal L}oc$; the category of local objects.
If ${\cal L}'\subset {\cal L}$ then $\overline{\cal L'}\subset \overline{\cal L}$ so one has always ${\cal L}oc\subset \overline{\cal L}$.

\begin{defi}\label{niscontext}
The data ${\cal N}=({\cal C=(A,B)},{\cal L})$ where $\cal C=(A,B)$ is a factorisation system and $\cal L$ a class of objects of $\cal C$ is called a {\em Nisnevich context}. If ${\cal L}=\overline{\cal L}$ the Nisnevich context is said {\em saturated}.
$\overline{\cal N}=({\cal C=(A,B)},\overline{\cal L})$ is called the {\em saturation} of $\cal N$. Two Nisnevich contexts are said to be {\em equivalent} if they have the same saturation (which implies in particular that they have the same underlying factorisation system).
\end{defi}


\begin{lemma}
For a Nisnevich context ${\cal N}=({\cal C=(A,B)},{\cal L})$, if a map $L\to L'\in \cal A$ is such that $L\in \overline{\cal L}$ then $L'\in \overline{\cal L}$.
\end{lemma}
\begin{proof}
If $\{U_i\to X\}$ is a Nisnevich covering family, by hypothesis any map $L\to X$ lift through one of the $U_i\to X$. If the map $L\to X$ is coming from a map $L'\to X$, this give a lifting square and a map $L'\to U_i$.
\end{proof}

\begin{lemma}
For a Nisnevich context ${\cal N}=({\cal C=(A,B)},{\cal L})$ and $X\in {\cal C}$, ${\cal N}_X=({\cal C}/X=({\cal A}_X,{\cal B}_X),{\cal L}/X)$, where ${\cal L}/X$ is the class of maps from an objet of ${\cal L}$ to $X$, is a Nisnevich context. Moreover, if ${\cal N}$ is saturated, so is ${\cal N}_X$.
\end{lemma}
\begin{proof}
The factorisation system $({\cal A}_X,{\cal B}_X)$ is the dual of that of lemma \ref{relativefacto}. Only the assertion about saturation is to be proven: the source of a local object of ${\cal N}_X$ is a local object of ${\cal N}$.
\end{proof}

\subsection{Distinguished coverings}

If $\cal C$ was small, the category $\widetilde{\cal C}$ of sheaves on $\cal C$ for a Nisnevich topology would be a topos and by lemma \ref{localobjects} a Nisnevich-local object would be exactly a point of $\widetilde{\cal C}$.
To have a small category to replace ${\cal C}$, we can consider ${\cal C}^\omega$ and the topos $\widetilde{{\cal C}^\omega}$ or more generally, given $X\in{\cal C}$, ${\cal C}^f/X$ and the topos $\widetilde{{\cal C}^f/X}$.

Points of $\widetilde{{\cal C}^f/X}$ are representable in $Pro({\cal C}^f/X)\simeq {\cal C}/X$ and it is easy to check that any local object defines such a point, but the converse may not be true.

\begin{defi}
A Nisnevich context is said to be {\em compatible} if, for any $X\in {\cal C}$, the category $\overline{\cal L}/X$ is the category of points of the topos $\widetilde{{\cal C}^f/X}$. 
\end{defi}

\begin{defi}
Given a Nisnevich context, a {\em distinguished class of Nisnevich covering families} is defined, for any $X\in {\cal C}$, as a class of Nisnevich covering families $U_i\to Y$ in ${\cal C}^f/X$ such that an object $L\in {\cal C}/X$ is Nisnevich local iff it lifts through any distinguished Nisnevich covering family.
\end{defi}

The following lemma is a reformulation of the definition of a compatible Nisnevich context.
\begin{lemma}
A Nisnevich context is compatible iff there exists a distinguished class of Nisnevich covering families.
\end{lemma}

\begin{lemma}
If $({\cal C=(A,B)},{\cal L})$ is a compatible Nisnevich context iff, for any $X\in {\cal C}$, $({\cal C}/X=({\cal A}_X,{\cal B}_X),{\cal L}/X)$ is a compatible Nisnevich context.
\end{lemma}
\begin{proof}
By definition.
\end{proof}

\medskip
The following definition will be used in theorem \ref{theorem}.
\begin{defi}\label{good}
A Nisnevich context $({\cal C=(A,B)},{\cal L})$ is said to be {\em good} if
\begin{enumerate}[a.]
\item ${\cal C}^o$ is locally finitely presentable,
\item ${\cal (A,B)}$ is right generated by maps in ${\cal B}^f$,
\item and it is compatible.
\end{enumerate}
\end{defi}

Hypothesis b. implies by corollary \ref{indleftgen} that, for any $X$, ${\cal B}{/X}\simeq Pro({\cal B}^f{/X})$.

\newpage

\subsection{Spectra}\label{spectra}

Let $\cal C$ be the opposite of a locally finitely presentable category and ${\cal N}=({\cal C=(A,B)},{\cal L})$ be a Nisnevich context.

\begin{defi}
\begin{enumerate}[a.]
\item ${\cal B}^f{/X}$ endowed with the Nisnevich topology is called the {\em small site} of $X$. The associated topos is noted $Spec_{\cal N}(X)$ and called the {\em small ${\cal N}$-spectrum} of $X$.
\item ${\cal C}^f{/X}$ endowed with the Nisnevich topology is called the {\em big site} of $X$. The associated topos is noted $SPEC_{\cal N}(X)$ and called the {\em big ${\cal N}$-spectrum} of $X$.
\end{enumerate}
\end{defi}

\medskip
Let ${\cal T}opos$ be the category whose objects are toposes and morphisms equivalence classes of geometric morphisms for natural isomorphisms.

\begin{lemma}\label{fibprodB}
If $\cal C$ as finite limits, the category ${\cal B}^f{/X}$ has all finite limits and for $u:X\to Y\in{\cal C}$, the base change functor $u^*:{\cal B}^f{/Y}\to {\cal B}^f{/X}$ is left exact.
\end{lemma}
\begin{proof}
As ${\cal B}^f{/X}$ has a terminal object, it is sufficient to prove that is has fiber products. But, using the cancellation property as in thorem \ref{factolim}, they can be computed independently of the base $X$ in $\cal C$ (which will also imply the exactness of $u$) and as $\cal B$ and ${\cal C}^f$ are stable by pullback the resulting diagram it is in ${\cal B}^f$.
\end{proof}

We recall that a geometric morphism $u:\cal E\to F$ is {\em etale} ({\em local homeomorphism} in \cite[C.3.3.4]{elephant}) iff there exists an $F\in\cal F$ and an isomorphism ${\cal F}/F\simeq {\cal E}$ such that $u$ is equivalent to the geometric morphism ${\cal F}/F\tto {\cal F}$.

\begin{thm}\label{petitgros}
$Spec_{\cal N}(-)$ and $SPEC_{\cal N}(-)$ are functors ${\cal C}\to {\cal T}opos$.
Moreover, maps in ${\cal B}^f$ are send to etale maps of toposes.
\end{thm}
\begin{proof}
We detail only the functoriality of the small spectrum. 
A map $u:X\to Y\in \cal C$ induces a base change functor $u^*:{\cal B}^f{/Y}\to {\cal B}^f{/X}$ that is left exact by lemma \ref{fibprodB} and clearly preserve covering families, so it is continuous \cite[III.1.6]{SGA4-1} and defines a geometric morphism $(u^*,u_*):Spec_{\cal N}(X)\tto Spec_{\cal N}(Y)$.
Given another map $v:Y\to Z$, the functors $(vu)^*$ and $u^*v^*:{\cal B}^f{/Z}\to {\cal B}^f{/X}$ are isomorphic so the associated geometric morphisms agree in ${\cal T}opos$.
As for the second statement: for any $X\to Y\in{\cal B}^f$, $Spec_{\cal N}(X)\simeq Spec_{\cal N}(Y)/X$.
\end{proof}

\begin{prop}
For $X\in {\cal C}$, if ${\cal C}^f{/X}$ is small, there exists two geometric morphisms (natural in $X$) $r_X=(r^*_X,r_*^X):SPEC_{\cal N}(X)\to Spec_{\cal N}(X)$ and $s_X=(s^*_x,s_*^X):Spec_{\cal N}(X)\to SPEC_{\cal N}(X)$, such that 
\begin{itemize}
\item $r_*^X=s^*_X$,
\item $r^*_X$ and $s^X_*$ are fully faithful, in particular $rs\simeq id$.
\end{itemize}
In other terms
\begin{itemize}
\item $r_X$ is left adjoint to $s_X$ in the bicategory of toposes,
\item $r_X$ is a quotient with connected fiber, 
\item and $s_X$ is a subtopos embedding and a section of $r_X$, \ie the adjunction $(r_X,s_X)$ is a reflexion of $SPEC_{\cal N}(X)$ on $Spec_{\cal N}(X)$.
\end{itemize}
\end{prop}
\begin{proof}
The morphism of small sites $\iota_X:{\cal B}^f{/X}\to {\cal C}^f{/X}$ commute to finite limits, and the topology of ${\cal B}^f$ is induced by that of ${\cal C}^f{/X}$, so $\iota$ is continous and cocontinuous by \cite[III.3.4]{SGA4-1}, and induces three adjoint functors $\iota_!^X\dashv \iota^*_X\dashv\iota_*^X$:
$$\xymatrix{
SPEC_{\cal N}(X)\ar@<-1ex>[r]_{\iota_*^X}\ar@<2.5ex>[r]^{\iota_!^X}& Spec_{\cal N}(X)\ar[l]_{\iota^*_X}.
}$$
$\iota_X$ being fully faithful, so are $\iota_!^X$ and $\iota_*^X$.
$r_X$ is defined as the adjonction $(\iota^*_X,\iota_*^X)$ and $s_X$ is defined as the adjonction $(\iota_!^X,\iota^*_X)$. For $s_X$ to be a geometric morphism, we need to check that $\iota_!^X$ is left exact, but this is a consequence of $\iota$ being left exact.
\end{proof}

\begin{cor}\label{souspoint}
The category of points of $Spec_{\cal N}(X)$ is a reflexive full subcategory of that of $SPEC_{\cal N}(X)$.
\end{cor}

\bigskip
We study now the functoriality of our spectra with respect to the factorisation system. We are going to focus only on $Spec$ but the results are the same for $SPEC$.
A unique factorisation systems on $\cal C$ is entirely characterized by its right classe $\cal B$. It is then possible to put an order of them by looking at the inclusion relation of these classes. 
\begin{defi}
For two factorisation systems $({\cal A}_i,{\cal B}_i),\ i=1,2$ on $\cal C$, we say that $({\cal A}_1,{\cal B}_1)$ is {\em finer than} or {\em a refinement of} $({\cal A}_2,{\cal B}_2)$ if ${\cal B}_2\subset {\cal B}_1$. This order admit an initial and a terminal element that are detailed in \S\ref{extremal}.
More generally, a Nisnevich context ${\cal N}=({\cal C}=({\cal A}_1,{\cal B}_1),{\cal L}_1)$ will be said {\em finer than} (or {\em a refinement of}) ${\cal N}'=({\cal C}=({\cal A}_2,{\cal B}_2),{\cal L}_2)$ if the underlying factorisation systems are the same, if $({\cal A}_1,{\cal B}_1)$ is finer than $({\cal A}_2,{\cal B}_2)$ and if ${\cal L}_1\subset\overline{\cal L}_2$.
\end{defi}

\begin{prop}\label{refineniscontext}
For two Nisnevich contexts ${\cal N}_i,\ i=1,2$, if ${\cal N}_1$ is refinement of ${\cal N}_2$, there is a natural transformation of functors $Spec_{{\cal N}_1}(-)\to Spec_{{\cal N}_2}(-)$.
\end{prop}
\begin{proof}
${\cal B}_2\subset{\cal B}_1$ so ${\cal P}t_{{\cal B}_1^f}\subset {\cal P}t_{{\cal B}_2^f}$, ${\cal B}_2$ point covering families are ${\cal B}_1$ point covering families. The functor
$$
{{\cal B}_2}{/X} \tto {{\cal B}_1}{/X}
$$
is then continuous and gives a geometric morphism
$$
Spec_{{\cal N}_1}(X)\tto Spec_{{\cal N}_2}(X),
$$
\ie $Spec$ is covariant with respect to the refinement relation for factorisation systems.
As for the Nisnevich forcing, ${{\cal B}_2}{/X} \tto {{\cal B}_1}{/X}$ will send covering families to covering families iff the forcing class ${\cal L}_1$ is contained in $\overline{\cal L}_2$ and the variance is the same.
\end{proof}

\subsubsection{Moduli interpretation}\label{moduliinterpretation}

We investigate a computation of the categories of points of the two spectra. Theorem \ref{theorem} establishes that, if the Nisnevich context is good, they can be described as local objects.
A complete study of the moduli aspects of our spectral theory would ask to compute not only global points but all categories of points of our spectra with values in any topos; this would require to develop more the topos theoretic aspects which we'll do in another work \cite{anel1}.

\medskip
If $P\to X$ is a point of an object $X$, we already interpreted the factorisation $P\to L\to X$ as the germ of the point in $X$.
This suggest the following definitions.
A {\em local form} of an object $X$ is a map $L\to X\in{\cal B}$ where $L$ is a local object, it is {\em pointed} if $L$ is.
Any point of $X$ define a pointed local form of $X$.
Let ${\cal L}oc(X)$ be the full subcategory of ${\cal C}{/X}$ generated by local forms of $X$, the left cancellation property of $\cal B$ (proposition \ref{proplift}) ensures that all morphisms between local forms of $X$ are in $\cal B$.
More generally, for a Nisnevich forcing class $\cal L$ with saturation $\overline{\cal L}$, a {\em ${\cal L}$-local form} of $X$ is a map $L\to X\in\cal B$ where $L\in \overline{\cal L}$ and the category $\overline{\cal L}(X)$ of ${\cal L}$-local forms of $X$ is defined as the subcategory of $\overline{\cal L}{/X}$ generated by objects whose structural map is in $\cal B$. Again, all morphisms of $\overline{\cal L}(X)$ are in $\cal B$.

\medskip
Let's recall the characterization of points of a site.
\begin{prop}\label{pointsite}
Let ${\cal D}$ be a site with a topology given by some covering families, the category of points of the associated topos $\widetilde{\cal D}$ is the full subcategory of $Pro({\cal D})$ of those pro-objects of ${\cal D}$ that have the lifting property through any covering family.
\end{prop}
\begin{proof}
Briefly (see \cite{MM} for details), the category of points of $\widehat{\cal D}$ is $Pro({\cal D})$ the category of pro-objects of ${\cal D}$. In $Pro(D)$, an object $P$ is a point of $\widetilde{\cal D}\subset \widehat{\cal D}$ iff it transforms covering families into epimorphic families. This last part is equivalent to have in $Pro({\cal D})$ a diagram
$$\xymatrix{
&U_i\ar[d]\\
P\ar[r]\ar@{-->}[ru]^{\exists i} &X
}$$
hence the statement of the result.
\end{proof}

\begin{thm}\label{pointspec}\label{theorem}
For a good Nisnevich context ${\cal N}=({\cal C=(A,B)},{\cal L})$ (def. \ref{good}):
\begin{enumerate}
\item the category of points of $SPEC_{\cal N}(X)$ is that $\overline{\cal L}{/X}$ of local objects over $X$
\item and the category of points of $Spec_{\cal N}(X)$ is that $\overline{\cal L}(X)$ of $\cal L$-local forms of $X$.
\end{enumerate}
\end{thm}
\begin{proof}
The first assertion is the hypothesis of compatibility on $\cal N$. As for the second one, by corollary \ref{souspoint}, local objects over $X$ will define points of $Spec_{\cal N}(X)$ iff they are in $Pro({\cal B}^f{/X})$, or $Pro({\cal B}^f{/X})\simeq{\cal B}{/X}$ by b. of definition \ref{good}, and points of $Spec_{\cal N}(X)$ are exactly local forms of $X$.
\end{proof}

\medskip
To finish, we recall 
that a topos is said {\em spatial} if can be written as the topos of sheaves of a topological space. The category of points of such a topos is at most a poset, this remark will be used to prove that most of our examples of spectra are not spaces.

\subsubsection{Structure sheaf}\label{structuresheaf}

For $u:X\to Y\in {\cal C}$, the naturality of $s_X$ and $r_X$ gives a diagram
$$\xymatrix{
Spec_{\cal N}(X)\ar[r]^-{s_X}\ar[d]_u\ar@{}[rd]|{(1)}&SPEC_{\cal N}(X)\ar[d]^U\ar[r]^-{r_Y}\ar@{}[rd]|{(2)}&Spec_{\cal N}(X)\ar[d]^u\\
Spec_{\cal N}(Y)\ar[r]_-{s_Y}&SPEC_{\cal N}(Y)\ar[r]_-{r_Y}&Spec_{\cal N}(Y)
}$$
of which we are going to study the commutation properties.

\begin{prop}\label{commutation}
The square (2) is commutative up to a natural isomorphism, and the square (1) up to a natural transformation $\alpha(u):Us_X\to s_Yu$.
Moreover, under the hypothesis of theorem \ref{pointspec}, for each point ${\cal S}\to Spec_{\cal N}(X)$ the morphism induced by $\alpha(u)$ on points of $SPEC_{\cal N}(Y)$ is in $\cal A$ (see proof).
\end{prop}
\begin{proof}
For the square $(2)$, it is sufficient to check it at the level of the inverse image functors restricted to the generating sites and it is a consequence of the stability of ${\cal B}^f$ by pullback in ${\cal C}^f$.
The result on $(1)$ can be deduced: there is a natural isomorphism $r_YUs_X\simeq r_Ys_Yu (\simeq u)$, composing by $s_Y$ and using the unit and counit of the adjunction $(r_Y,s_Y)$, we obtain the wanted map $\alpha(u):Us_X\to s_Yu$.

\medskip
For the second part, points of a topos can be viewed as some pro-objects and the effect on points of a geometric morphism $(u^*,u_*):{\cal E}\to {\cal F}$ is understood looking at the left pro-adjoint $u_!$ of $u^*$. 
If $\cal E$ is a topos, the category $\underline{Pro}(\cal E)$ of internal pro-objects of $\cal E$ is defined as the category of $\cal E$-enriched left-exact accessible endofunctors of $\cal E$. In particular, it contains fully faithfully the category $Pro(\cal E)$ of pro-objects of $\cal E$ view as a category 

Given a geometric morphism $u:{\cal E}_1\to {\cal E}_2$, ${\cal E}_1$ can be enriched over ${\cal E}_2$ by defining the enriched Hom to be $u_*\uHom_{{\cal E}_1}(x,y)$ where $\uHom_{{\cal E}_1}$ is the internal Hom of ${\cal E}_1$. For this enrichement, both $u^*$ and $u_*$ are enriched functors. 

For a geometric morphism $u:{\cal E}\to {\cal F}$, the left pro-adjoint of $u^*:{\cal F}\to {\cal E}$ is defined the following way: every object $X\in{\cal E}$ define a geometric morphism $i_X=(i_X^*,i^X_*):{\cal E}{/X}\to {\cal E}$, and by composition an endofunctor $u_*i_*^Xi^*_Xu^*$ of ${\cal F}$, this endofunctor is $\cal F$-enriched left exact as a composition of such functors so is copresentable by an internal pro-object $u_!(X)$ of $F$. This construction is functorial in $X$ and define a functor $u_!:{\cal E}\to \underline{Pro}({\cal F})$. As for the adjunction property:
\begin{eqnarray*}
X&\tto & u^*Y\\
X\simeq i^*_X(X)&\tto & i^*_Xu^*Y\\
*\simeq i_*^X(X)&\tto & i_*^Xi^*_Xu^*Y\\
*\simeq u_*(*)&\tto & u_*i_*^Xi^*_Xu^*Y \simeq \underline{Hom}_{\underline{Pro}({\cal F})}(u_!(X),Y)
\end{eqnarray*}
where $\underline{Hom}_{\underline{Pro}({\cal F})}(-,-)$ is the $\cal F$-enriched hom of $\underline{Pro}({\cal F})$.

\medskip
We will now compute the pro-adjoints of the following diagram and their action on the categories of points:
$$\xymatrix{
\underline{Pro}(\widetilde{{\cal B}^f{/X}})\ar[r]^-{\check s^X_!}\ar[d]_-{\check u_!}& \underline{Pro}(\widetilde{{\cal C}^f{/X}})\ar[d]^-{\check U_!}\\
\underline{Pro}(\widetilde{{\cal B}^f{/Y}})\ar[r]^-{\check s^Y_!}&\underline{Pro}(\widetilde{{\cal C}^f{/Y}})
}$$
where, for a geometric morphism $(u^*,u_*):{\cal E}\to {\cal F}$, $\check u_!:\underline{Pro}({\cal E})\to \underline{Pro}({\cal F})$ denotes the (internal) right Kan extension of $u_!$. $\check u_!$ is left adjoint to the right Kan extension $\check u^*$ of $u^*$.
The diagram is still commutative up to a natural transformation constructed the same way as before (in a sense this is the same natural transformation).

To extract the action on points we'll use implicitly the following lemma.
\begin{lemma}\label{proadj}
If in a diagram of functors
$$\xymatrix{
C\ar[r]^\gamma\ar@<-.5ex>[d]_v&C'\ar@<-.5ex>[d]_{v'}\\
D\ar[r]_\delta\ar@<-.5ex>[u]_u&D'\ar@<-.5ex>[u]_{u'}\ ,
}$$
$v$ is left adjoint to $u$, $v'$ left adjoint to $u'$, $\gamma$ and $\delta$ are dense in the sense that any object of $C'$ (resp. $D'$) is a limit of objects of $C$ (resp. $D$) and $\gamma u=u'\delta$, then $\delta v=v'\gamma$, \ie $v$ is the restriction of $v'$ to $C$.
\end{lemma}
\begin{proof}
Any $y\in D'$ can be written $y=\lim_i\delta(y_i)$, so for all $x\in C,y\in D$: $D'(\delta v(x),y)\simeq \lim_i D'(\delta v(x),\delta(y_i)) \simeq \lim_i C(x,u(y_i))\simeq \lim_i C'(\gamma(x),\gamma u(y_i))\simeq \lim_i C'(\gamma(x),u'\delta(y_i))\simeq \lim_i D'(v'\gamma(x),\delta(y_i))\simeq D'(v'\gamma(x),y)$.
\end{proof}

The functor $\check s^X_!$ is the extension of the inclusion ${\cal B}^f{/X}\to {\cal C}^f{/X}$, so we have a diagram:
$$\xymatrix{
\overline{\cal L}(X)\ar[r]\ar[d]& {\cal B}{/X}\ar[r]^-{\simeq}\ar[d]&Pro(\widetilde{{\cal B}^f{/X}})\ar[r]&\underline{Pro}(\widetilde{{\cal B}^f{/X}})\ar[d]^-{\check s^X_!}\\
\overline{\cal L}{/X}\ar[r]&{\cal C}{/X}\ar[r]^-{\simeq}&Pro(\widetilde{{\cal C}^f{/X}})\ar[r]&\underline{Pro}(\widetilde{{\cal C}^f{/X}})\ ,
}$$
where the horizontal arrows are fully faithful and the vertical arrows are all restrictions of $\check s^X_!$. The morphism induced on points is simply the inclusion of $\overline{\cal L}(X)$ in $\overline{\cal L}{/X}$. The result is analog for $s_!^Y$.

For $\check U_!$ we have a diagram
$$\xymatrix{
\overline{\cal L}{/X}\ar[r]\ar[d]& {\cal C}{/X}\ar[r]^-{\simeq}\ar[d]&Pro(\widetilde{{\cal C}^f{/Y}})\ar[r]&\underline{Pro}(\widetilde{{\cal C}^f{/X}})\ar[d]^-{\check U_!}\\
\overline{\cal L}{/Y}\ar[r]&{\cal C}{/Y}\ar[r]^-{\simeq}&Pro(\widetilde{{\cal C}^f{/Y}})\ar[r]&\underline{Pro}(\widetilde{{\cal C}^f{/Y}}).
}$$
$u^*:{\cal C}{/Y}\to {\cal C}{/X}$ has a left adjoint $u_!$ given by composing with $u$, which is the restriction of $\check U_!$.

For $\check u_!$ we have a diagram
$$\xymatrix{
\overline{\cal L}(X)\ar[r]\ar[d]& {\cal B}{/X}\ar[r]^-{\simeq}\ar[d]^{\upsilon}&Pro(\widetilde{{\cal B}^f{/X}})\ar[r]&\underline{Pro}(\widetilde{{\cal B}^f{/X}})\ar[d]^-{\check u_!}\\
\overline{\cal L}(Y)\ar[r]& {\cal B}{/Y}\ar[r]^-{\simeq}&Pro(\widetilde{{\cal B}^f{/Y}})\ar[r]&\underline{Pro}(\widetilde{{\cal B}^f{/Y}})\\
}$$
We prove that the functor $u^*=-\times_YX:{\cal B}{/Y}\to {\cal B}{/X}$ has a left adjoint given by sending $b:U\to X$ to the $\phi(ub)\to Y$ where $U\to \phi(ub)\to Y$ is the factorisation of $ub:U\to X\to Y$: given a choice of $\cal (A,B)$ factorisation for any arrow of $\cal C$, a map $b:U\to X$ defines a unique square
$$\xymatrix{
U\ar[r]^{\alpha}\ar[d]_b& \phi(ub)\ar[d]^{\beta}\\
X\ar[r]^u&Y
}$$
where $U\to \phi(ub)\to Y$ is defined as the factorisation of $ub:U\to Y$.
From this we deduce a bijection between the set of squares
$$(*)=\begin{array}{c}
\xymatrix{
U\ar[r]\ar[d]& V\ar[d]\\
X\ar[r]^u&Y
}\end{array}
$$
where $V\to Y\in {\cal B}$ and that of morphisms of ${\cal B}{/Y}$:
$$\xymatrix{
\phi(ub)\ar[r]\ar[dr]& V\ar[d]\\
&Y
}$$
(to prove the bijection, the map $\phi(ub)\to V$ comes from a lifting condition).
But squares $(*)$ are also in bijection with morphisms in ${\cal B}{/X}$:
$$\xymatrix{
U\ar[r]\ar[d]& V\times_YX=u^*V\ar[dl]\\
X
}$$
which gives us the adjonction.
Now, the restriction to $\overline{\cal L}(X)$ takes its values in $\overline{\cal L}(Y)$ and is the morphism induced by $u$ between the categories of points.

Finally the situation is the following: a point $b:L\to X$ is send on one side to $ub:L\to Y$ and on the other to $\beta:\phi(ub)\to Y$ and the natural transformation $\alpha(u)$ is given by the factorisation
$$\xymatrix{
L\ar[r]^-{\alpha(u)}\ar[rd]_{bu}&\phi(ub)\ar[d]^{\beta}\\
&Y.
}$$
This is what we meant saying that it was given by a map in $\cal A$.
\end{proof}

If $*$ is the terminal object of $\cal C$, the category of points of the topos $SPEC_{\cal N}(*)$ is $\cal C$.

\begin{defi}
The composition $O_X^{\cal N}:Spec_{\cal N}(X)\to SPEC_{\cal N}(X)\to SPEC_{\cal N}(*)$ is called the {\em structural sheaf} of $X$.
For every point $x:{\cal S}\to Spec_{{\cal B}^f}(X)$, the {\em stalk of ${\cal O}_X^{\cal N}$ at $x$} is the object of $\cal C$ induced point ${\cal O}_{X,x}^{\cal N}:{\cal S}\to SPEC_{\cal N}(*)$.
\end{defi}

\begin{prop}
For a point $x:{\cal S}\to Spec_{\cal N}(X)$ corresponding to a local form $L\to X$, the stalk $O^{\cal N}_{X,x}$ is the objet $L$.
\end{prop}
\begin{proof}
$O^{\cal N}_X$ is the composition $Spec_{\cal N}(X)\to SPEC_{\cal N}(X)\to SPEC_{\cal N}(*)$ and the action of these morphisms on the points have been explained inside the proof of proposition \ref{commutation}.
\end{proof}

As a corollary of proposition \ref{commutation}, a map $u:X\to Y\in {\cal C}$ induces a diagram of toposes
$$\xymatrix{
Spec_{\cal N}(X)\ar[r]^{s_X}\ar[d]_u&SPEC_{\cal N}(X)\ar[d]^U\ar[rd]\\
Spec_{\cal N}(Y)\ar[r]_{s_Y}&SPEC_{\cal N}(Y)\ar[r]&SPEC_{\cal N}(*)
}$$
and $\alpha(u)$ induces a natural transformation ${\cal O}(u):{\cal O}_X\to {\cal O}_Y\circ u$ such that, for every point of $x:{\cal S}\to Spec_{{\cal B}^f}(X)$, the induced map on the stalk ${\cal O}(u)_x:{\cal O}_{X,x}\to {\cal O}_{Y,u(x)}$ is in $\cal A$.
This result is an analog of the fact that a map of rings $A\to B$ induces a local morphism between the local rings corresponding to the stalks of the structural sheaves of $Spec(A)$ and $Spec(B)$.

\bigskip
The category of points of $SPEC_{\cal N}(*)$ is that $\overline{\cal L}$ of local objects and the factorisation system of $\cal C$ restrict to $\overline{\cal L}$. This is in fact a general phenomenon and for every topos $\cal T$ the category of morphisms from $\cal T$ to $SPEC_{\cal N}(*)$ will inherit a unique factorisation system. The point of view chosen for the exposition in this work makes the details of this factorisation system a bit complicated to explicit and we won't explain this here. We won't explain either the nice adjunction property of the small spectrum implying that it is a universal localisation. We will treat these questions in a better context in \cite{anel1}.

\bigskip
We study now the functoriality of the map $Spec_{\cal N}(X)\to SPEC_{\cal N}(X)$ with respect to the Nisnevich context.
\begin{prop}\label{changenis}
For two Nisnevich contexts ${\cal N}=({\cal C}=({\cal A}_1,{\cal B}_1),{\cal C}^f,{\cal L}_1)$ and ${\cal N}'=({\cal C}=({\cal A}_2,{\cal B}_2),{\cal C}^f,{\cal L}_2)$, if $\cal N$ is finer than $\cal N'$, there is a diagram of geometric morphisms
$$\xymatrix{
Spec_{\cal N}(X)\ar[r]^-{s_X}\ar[d]_r\ar@{}[rd]|{(1)}&SPEC_{\cal N}(X)\ar[d]^R\ar[r]^-{r_X}\ar@{}[rd]|{(2)}&Spec_{\cal N}(X)\ar[d]^r\\
Spec_{\cal N'}(X)\ar[r]_-{s'_X}&SPEC_{\cal N'}(X)\ar[r]_-{r'_X}&Spec_{\cal N'}(X)
}$$
where $R$ is a subtopos embedding, 
$(2)$ commutes up to a natural isomorphism
and $(1)$ commutes up to a natural transformation $\beta$.
At the level of the category of points, $\beta$ is given by a map in ${\cal B}_1\cap {\cal A}_2$.
\end{prop}
\begin{proof}
The assertion about $R$ is a consequence of the facts that $SPEC_{\cal N}(X)$ and $SPEC_{\cal N'}(X)$ have the same underlying site and the topology of $SPEC_{\cal N}(X)$ is finer.
The commutation of $(2)$ can be seen at the level of inverse images restricted to the sites.
From it we deduced natural isomorphisms $rr_Xs_X\simeq r_XRs_X$ and composing by $s'_X$ and using the co-unit of $(r_X,s_X)$ and the unit of $(r'_X,s'_X)$ we have a transformation $\beta:Rs_X\to s'_Xr$.
As for the action of $\beta$ on points we need only to study the pro-adjoint $r_!$.
We'll use again lemma \ref{proadj}. By an analog argument to that in the proof of proposition \ref{commutation}, the map $({\cal B}_2){/X}\to ({\cal B}_1){/X}$ admits a left adjoint given by the $({\cal A}_2,{\cal B}_2)$ factorisation:
$$\xymatrix{
U\ar[rd]_b\ar[r]^\beta&\phi(b)\ar[d]^{\beta'}\\
&X
}$$
$\beta\in {\cal A}_2$ by definition, and as both $b$ and $\beta'$ are in ${\cal B}_1$, so is $\beta$ by cancellation.
The map $\overline{\cal L}_1(X)\to \overline{\cal L}_2(X)$ between the categories of points is then given also by this factorisation.
\end{proof}

This result will be used in particular when ${\cal N}'$ is the Indiscrete factorisation context (\S\ref{extremal}) to defined the structural map of the structural sheaf.


\section{Examples}\label{examples}

This part deals with examples of the previous setting.
After a short part on the two trivial factorisation systems that always exist on a category, we present how Zariski and Etale topology are associated to unique factorisation systems according to the scheme of the previous section and how the general notion of point and local objects, gives back known classes of objects. The Nisnevich topology is also considered as an illustration and a motivation of the Nisnevich forcing.

Then, what is more interesting, we study a sort of dual systems where Zariski closed sets play the role of opens and finite maps that of etale maps. There is also a notion of Nisnevich topology in this context. This material has some flavour of Voevodsky cdh topologies and, again, the general framework extract known classes of objects. Section \ref{duality} contains some remarks about these two dual settings, but raises more question than it gives answers.

The last section study very rapidly example of factorisation systems outside our general setting and outside of algebraic geometry but where some of our construction can be adapted.

\subsection{Extremal examples}\label{extremal}

Every category $\cal C$ admits a two canonical unique factorisation systems ${\cal C}=(Iso({\cal C}),{\cal C})$ and ${\cal C}=({\cal C},Iso({\cal C}))$ where $Iso(\cal C)$ is the subcategory of isomorphisms. The factorisation of a map is then given by composing with the identity of the source or of the target.
These two systems will be called respectively {\em discrete} and {\em indiscrete} because they behave like the discrete and indiscrete topologies, being somehow the finest and the coarsest factorisation systems. 

We assume $\cal C$ is the opposite of a locally finitely presentable category.

\paragraph{Discrete factorisation system}
${\cal C}=(Iso({\cal C}),{\cal C})$ is the discrete factorisation system. 
Points are objets $P$ splitting every map $U\to P$, their full subcategory in $\cal C$ is a groupoid. Little can be said in general, beside that they will be points of any factorisation system on $\cal C$. Little can be said also about covering families or local objects. The only remark is that the small and big toposes agree in this case (and are noted $SPEC_{Dis}(X)=Spec_{Dis}(X)$).

The Nisnevich context $Dis=({\cal C}=(Iso({\cal C}),{\cal C}),\emptyset)$ is the finest Nisnevich context.

\medskip
In the case where ${\cal C}=CRings^o$, the opposite category of that of commutative rings, the set of discrete points is empty: it would be the set of rings $A$ such that any map $A\to B$ has a retraction, but only the zero ring as this property and it is excluded from points by definition. This imply that the set of points of any object is empty and that the empty family will cover any object, collapsing both $Spec_{Dis}$ and $SPEC_{Dis}$ to the empty topos.

\paragraph{Indiscrete factorisation system}

${\cal C}=({\cal C},Iso({\cal C}))$ is the indiscrete factorisation system.
Every object is a point, hence local, and the factorisation topology is trivial. The Nisnevich context $Ind=({\cal C}=(Iso({\cal C}),{\cal C}),\emptyset)$ is the coarsest Nisnevich context.
The small site of $X$ is reduced to a ponctual category and $Spec_{Ind}(X)$ is the ponctual topos and the big topos $SPEC_{Ind}(X)$ it is the topos of presheaves over ${\cal C}^f{/X}$, whose category of points is ${\cal C}/X$.

\medskip
The structural sheaf ${\cal S}\simeq Spec_{Ind}(X)\tto SPEC_{Ind}(*)$ of an object $X$ is simply $X$ view as a point of $SPEC_{Ind}(*)$.

\paragraph{Comparisons}

For any other Nisnevich context $\cal N$, proposition \ref{changenis} gives a diagram
$$\xymatrix{
Spec_{\cal N}(X)\ar[r]^-{O^{\cal N}_X}\ar[d]_{r_{Ind}}&SPEC_{\cal N}(X)\ar[d]^{R_{Ind}}\\
{\cal S}\ar[r]^-{X}&SPEC_{Ind}(*)
}$$
and a natural transformations $\beta_{Ind}:O^{\cal N}_X\to X\circ r_{Ind}$, called the {\em structural map} of the structure sheaf.

\begin{prop}
For a point $x:{\cal S}\to Spec_{\cal N}(X)$ corresponding to a local form $L\to X$,
the map $\beta_{Ind,x}:O^{\cal N}_{X,x}\to X\circ r_{Ind}\circ x$ viewed in ${\cal C}$ is that map $L\to X$.
\end{prop}
\begin{proof}
This is can be deduced from the proof of proposition \ref{changenis}.
\end{proof}


\subsection{Zariski topology}\label{zariski}

The category $\cal C$ is the opposite of that of commutative unital rings, but to simplify the manipulation we are going to work in ${\cal C}^o=CRings$. All definitions of points and local objects will have to be opposed, and the role of left and right class of maps are interchanged: the $(Loc,Cons)$ factorisation system that we'll construct on $CRings$ has to be though as $(Cons^o,Loc^o)$ in $CRings^o$.
We apologize to the reader for this inconvenience, but we felt that it was better to develop the general framework with the geometric intuition rather than the algebraic one.

\subsubsection{Factorisation system}

A map $A\to B$ in $CRings$ is called a {\em localisation} if there exists a set $S\in A$ and $B\simeq A[\{x_s,s\in S\}]/(\{sx_s-1,s\in S\})$.
The class of localisation maps is noted $Loc$.
A map $u:A\to B$ in $CRings$ is called a {\em conservative} if any $a\in A$ is invertible iff $u(a)$ is.
The class of conservative maps is noted $Cons$.

The following lemma is a reformulation of the definition.
\begin{lemma}\label{genloccons}
A map is conservative iff it has the right lifting property with respect to $\ZZ[x]\tto \ZZ[x,x^{-1}]$.
\end{lemma}

\begin{prop}
The classes of maps $Loc$ and $Cons$ are the left and right class of a unique factorisation system.
\end{prop}
\begin{proof}
For a map $u:A\to B$, we define $S:=u^{-1}(B^\times)$ and $A[S^{-1}]$ the associated localisation.
$u$ factors $A\to A[S^{-1}]\to B$, the first map is a localisation by construction, it remains to prove that $v:A[S^{-1}]\to B$ is conservative.
Let $a/s\in A[S^{-1}]$ such that $v(a/s)=u(a)u(s)^{-1}$ has an inverse $b\in B$, this is equivalent to the fact that $u(a)$ has an inverse, \ie to $a\in S$. Elements of $A[S^{-1}]$ invertible in $B$ are therefore fractions of elements of $S$, which are precisely the invertible elements of $A[S^{-1}]$.
To prove the unicity of the factorisation, we can use lemma \ref{uniquelift} as the codiagonal of a localisation map is always an isomorphism.
\end{proof}

\medskip
We'll use implicitly the following well-known result in the sequel.
\begin{lemma}
A localisation is of finite presentation iff it can be define by inverting a single element.
\end{lemma}
\begin{proof}
A localisation $A\to A[S^{-1}]$ is always the cofiltered colimit of $A\to A[F^{-1}]$ where $F$ run through all finite subsets of $S$. Now if $A\to A[S^{-1}]$ is of finite presentation, the identity of $A[S^{-1}]$ factors through one of the $A[F^{-1}]$ and this gives a section $s$ of $r:A[F^{-1}]\to A[S^{-1}]$. Now by cancellation both $s$ is an epimorphism and $srs=s$ implies also $sr=1$, so $r$ is an isomorphism. Finally if $F=\{f_1,\dots,f_n\}$, $A[F^{-1}]=A[(f_1\dots f_n)^{-1}]$.
\end{proof}

\subsubsection{Points}\label{zarpoints}

A {\em nilpotent extension} of a ring $A$ is a map $B\to A$ such that any element in the kernel is nilpotent.

The opposite of the condition for a point reads: a ring $A$ corresponds to a point iff it is non zero and for any non zero localisation $\ell:A\to A[a^{-1}]$ there exists $s$ a retraction of $\ell$.

\begin{prop}\label{zarclosed}
A ring $A$ corresponds to a point of the $(Cons^o,Loc^o)$ factorisation system iff it is a nilpotent extensions of a field.
\end{prop}
\begin{proof}
As a localisation is zero iff it inverses a nilpotent element of $A$, the condition of being a point says that any non nilpotent element of $A$ is invertible, so $A_{red}$ is a field.
\end{proof}

For short we are going to refer to these objects as {\em fat fields}. Any field is a fat field and the reduction of any fat field is a field.
Any fat field is a local ring, the unique maximal ideal being given by the nilradical.

\begin{prop}\label{zarpt}
The set of points of a ring $A$ is in bijection with the set of prime ideals of $A$.
\end{prop}
\begin{proof}
The set of points of $A$ is defined as a the set of all maps $A\to K$ with $K$ a fat field quotiented by the relation generated by $A\to K \sim A\to K'$ if there exists $K\to K'$ such that $A\to K\to K'=A\to K'$. Any $A\to K$ can be replaced by one where the target is a field ($A\to K_{red}$) and $K''$ above can always be taken to be a field too. This ensure that instead of fat fields one can use only fields to define the same set. The result is then classical: the kernel of a map to a field is a prime ideal and every prime ideal is the kernel of the map to its residue field.
\end{proof}

\subsubsection{Covering families}

It should be already clear that our covering families are exactly Zariski covering families but we'll need the following result to compute local objects.

\begin{prop}\label{zarfinprescov}
Finite presentation point covers are families $A\to A[a_i^{-1}]$ such that 1 is a linear combinaison of the $a_i$. As a consequence all point covering families admits a finite point covering subfamily.
\end{prop}
\begin{proof}
For $K$ a field, and a given $A\to K$, $a_i$ is either in the kernel of invertible in $K$, \ie $A\to K$ factors through $A[a_i^{-1}]$ or $A/a_i$.
So $A\to A[a_i^{-1}]$ is a cover iff no non zero $A\to K$ factors through $A/(a_i;i)$ iff $A/(a_i;i)=0$.
For the last equivalence if $A/(a_i;i)\not=0$ it has at least one residue field giving a map $A\to A/(a_i;i)\to K$ and if such a factorisation $A\to A/(a_i;i)\to K$ exists as $A\to K$ is non zero, $A/(a_i;i)$ has to be non trivial.
The conclusion is now deduced from $1\in (a_i;i)\iff A/(a_i;i)=0$.
\end{proof}

\subsubsection{Local objects}

A ring $B$ is a {\em local ring} iff for any $x,y\in B$ satisfying $x+y=1$ ($\iff x+y$ invertible), $x$ or $y$ is invertible.
This condition can be read as: a $A$-algebra $B$ is a local ring iff for any $x,y\in B$, the map $A[x,y,(x+y)^{-1}]\to B$ factors through $A[x,x^{-1}]$ or $A[y,y^{-1}]$. Now as 1 is a linear combinaison of $x$ and $y$ in $A[x,y,(x+y)^{-1}]$, the two maps $A[x,y,(x+y)^{-1}]\to A[x,x^{-1}]$ and $A[x,y,(x+y)^{-1}]\to A[y,y^{-1}]$ form a covering family and $B$ and this gives the following lemma.

\begin{lemma}\label{prooflocal}
A $A$-algebra $B$ is a local ring iff for any $x,y\in B$ such that $x+y$ is invertible, $B$ lift through the point covering family 
$A[x,y,(x+y)^{-1}]\to A[x,x^{-1}]$ and $A[x,y,(x+y)^{-1}]\to A[y,y^{-1}]$ of $A[x,y,(x+y)^{-1}]$.
\end{lemma}

\begin{prop}\label{zarpointablelocal}
A ring $A$ corresponds to a pointed local object for the $(Cons^o,Loc^o)$ system iff it is a local ring.
\end{prop}
\begin{proof}
In a local ring $(A,m)$, elements not in $m$ are invertible so $A\to A/m$ is a conservative map.
Conversely, let $u:A\to K$ be a conservative map with target a fat field, and $x,y\in A$ such that $x+y=1$, then the same equation holds in $K$ and $K$ being a local ring, either $u(x)$ or $u(y)$ is invertible in $K$. But $u$ being conservative the same is true in $A$.
\end{proof}

\begin{prop}\label{zarlocal}
A ring $A$ corresponds to a local object for the $(Cons^o,Loc^o)$ system iff it is a local ring.
\end{prop}
\begin{proof}
Any local ring is a local object by proposition \ref{zarpointablelocal}.
Now, let $A$ be a a local object and $x,y\in A$ such that $x+y=1$.
The family $\{A\to A[x^{-1}],A\to A[y^{-1}]\}$ is then a cover by proposition \ref{zarfinprescov} and the existence of a section of this cover says that either $x$ or $y$ is invertible in $A$.
\end{proof}

In this setting, the fact that pointed local and local objects coincide is a sophisticated way to say that any local ring has a residue field.

\subsubsection{Spectra and moduli interpretation}

It is clear that the topology given by the general theory coincide with the Zariski topology for affine schemes.
\begin{prop}
For $A\in CRings^o$, 
$Spec_{Zar}(A)$ is the usual small Zariski spectrum of $A$
and $SPEC_{Zar}(A)$ is the usual big Zariski topos of $A$.
\end{prop}

\begin{prop}
$SPEC_{Zar}(A)$ classifies $A$-algebras that are local rings, such algebras can have automorphisms so $SPEC_{Zar}(A)$ is not a spatial topos.
$Spec_{Zar}(A)$ classifies localisations of $A$ that are local rings. 
\end{prop}
\begin{proof}
We apply theorem \ref{pointspec}. The Nisnevich context $Zar=(CRings^o=(Cons^o,Loc^o),\emptyset)$ is good (definition \ref{good}): $CRings$ is locally finitely presentable, the rest follows by lemmas \ref{genloccons}
and \ref{prooflocal}.
\end{proof}

The following result is highly classical but not obvious from our definition.
\begin{prop}\label{zarspace}
$Spec_{Zar}(A)$ is a topological space.
\end{prop}
\begin{proof}
The topos $Spec_{Zar}(A)$ is generated by the category $(A/ Loc^f)^o$ which is a poset, so it is localic.
This poset is formed of compact objects and we would like to apply the result of \cite[II.3.]{stone} to deduced the local is coherent and then spatial. To do that we have to check that the topology on $(A/ Loc^f)^o$ is the jointly surjective topology.
First, $(A/ Loc^f)^o$ is a distributive lattice: the intersection of $A[a^{-1}]$ and $A[b^{-1}]$ is $A[ab^{-1}]$ and the union is the middle object $C$ of the $(Loc,Cons)$ factorisation of $A\to A[a^{-1}]\oplus A[b^{-1}]$ (indeed $C$ will add to $A$ all elements invertible both in $A[a^{-1}]$ and $A[b^{-1}]$, such $C$ will be some $A[c^{-1}]$); and to prove the distributive law the lemma is the following: if $B\to C\to D$ is a $(Loc,Cons)$ factorisation, for any $b\in B$, $B[b^{-1}]\to C[b^{-1}]\to D[b^{-1}]$ is still a $(Loc,Cons)$ factorisation, \ie $C[b^{-1}]\to D[b^{-1}]$ is still convervative but as new invertible elements in $B$ are fractions of denominator $b$ with invertible numerator, they can be lifted to $C[b^{-1}]$.

As for the topology on $(A/ Loc^f)^o$: for a finite family $a_i\in A$, $c\in A$ is invertible in all the $A[a_i^{-1}]$ iff $(a_i;i)\subset \sqrt{c}$, in particular there is an equivalence $(a_i;i)=A$ iff $c$ is invertible, so $A[a_i^{-1}]$ is a joint covering family iff $(a_i;i)=A$, which is also the characterisation of point covering families. The same reasoning work relatively to any $B\in (A/ Loc^f)^o$ and this proves that the factorisation topology is the jointly surjective one.
\end{proof}

Also in this case the two notions of points (of the factorisation system and of the spectrum) agree.
\begin{prop}\label{sameptzar}
For $A\in CRings$, the category of points of $Spec_{Zar}(A)$ is a poset equivalent to the opposite of that of prime ideals of $A$.
In particular the set of point of $Spec_{Zar}(A)$ is in bijection with $pt_{Zar}(A)$.
\end{prop}
\begin{proof}
We need to prove that this set is in bijection with that of prime ideals of $A$. This is well known: any prime ideal $p\subset A$ defines a point of $Spec_{Zar}(A)$ by $A\to A_{p}=A[(A\setminus p)^{-1}]$. And given a localisation of $A\to B$ where $B$ is a local ring, the inverse image of the maximal ideal of $B$ is a prime ideal $p$ of $A$ and $B\simeq A_{p}$. 
\end{proof}

\subsubsection{Remark on a variation}\label{variation}

A class $L$ of maps in a site $({\cal C},\tau)$ is said to be {\em local} if, for $u:X\to Y$, for any covering $V_i\to Y$ and any covering of $u_{ij}:U_{ij}\to V_i\times_YX$, the map $u$ is in $L$ iff all $u_{ij}$ are in $L$. Such classes are stable by intersection, so it is always possible to saturate any class $L$ into a local class $L^{loc}$ of maps locally (after pullback) maps in $L$.
If the class $L$ had moreover the property that covering sieves of $\tau$ can be generated by families of maps in $L$, it is clear that covering families in $L^{loc}$ will generate the same topology. 

\medskip
We claim that the class $Loc^o$ is not local for the Zariski topology on $CRings^o$ and its saturation is the class $Zet^o$ of etale maps that are locally trivial for the Zariski topology (called Zariski etale maps). We claim also that, remarkably, the class $Zet$ is again the left class of a unique factorisation system $(Zet,Conv)$ on $CRings$ where $Conv$ is the class of conservative maps having an extra unique lifting property for idempotents, \ie $(Zet,Conv)$ is left generated by $\ZZ[x]\to\ZZ[x,x^{-1}]$ and $\ZZ\to\ZZ[x]/(x^2-x)\simeq \ZZ\times\ZZ$. 

We could replace in the previous study the factorisation system $(Loc,Cons)$ by $(Zet,Conv)$ to generate the same factorisation topology and the same spectra, but with different sites. Only the proof of the spatiality of $Spec_{Zar}(X)$ would be less straightforward.

\subsection{Etale topology}\label{etale}

The category $\cal C$ is again $CRings^o$ and we keep the same convention of opposing everything as in the Zariski case.

\subsubsection{Factorisation system}

A map of rings is said {\em etale} if it flat and unramified \cite[\S3]{milne}.
The class of etale maps of finite presentation is noted $Et^f$, that of etale maps between rings of finite presentation is noted $Et_*^f$. A map of rings is said {\em henselian} if it has the right lifting property with respect to $Et_*^f$.
The class of henselian maps is noted $Hens$.

\begin{prop}
The class $^\bot Hens$ is the class $indEt:=Ind\z Et^f$ and the classes $indEt$ and $Hens$ are respectively the left and right classes of a unique factorisation system on $CRings$.
\end{prop}
\begin{proof}
We use theorem \ref{factolim} with $G=G'=Et_*^f$ and we have $\overline{G}=Et^f$.
\end{proof}

The factorisation of $A\to B$ is difficult to explicit but it consists in a separable closure of $A$ relatively to $B$: one needs to add an element to $A$ for every simple root in $B$ of every polynomial in $A[X]$.

\begin{lemma}\label{locsep}
\begin{enumerate}
\item $Hens\subset Cons$ and $Loc\subset indEt$.
\item $Loc^f$ point covering families are $Et^f$ point covering families.
\end{enumerate}
\end{lemma}
\begin{proof}
\textit{1.} From properties of lifting systems that the two inclusions are equivalent.
Any map lifting $u:\ZZ[X]\to \ZZ[X,X^{-1}]$ is conservative, so as $u$ is etale any henselian morphism in conservative.

\textit{2.} As $Loc\subset indEt$, points of the $(Hens^o,indEt^o)$ system are points of the $(Cons^o,Loc^o)$ system.
\end{proof}

Let $Nil$ be the class of maps in $CRings$ that are extensions by a nilpotent ideal. The class $Nil^\bot$ is the class $fEt$ of {\em formally etale maps} and if $\overline{Nil}$ is defined as $fEt^\bot$, $(fEt,\overline{Nil})$ is a unique lifting system that we are going to compare to $(Hens,indEt)$.

\begin{prop}\label{nilhens}
$\overline{Nil}\subset Hens$ and $indEt\subset fEt$ and those inclusions are strict.
\end{prop}
\begin{proof}
$indEt$ is the class of ind-etale maps of finite presentation, now as $fEt$ contains $Et^f$ and is stable by any colimit, $indEt\subset fEt$.

For the second point, it is enough to prove that $indEt\subset fEt$ is strict.
Let $A$ be a noetherian henselian local ring with residue field $k$ and $\widehat{A}$ its completion for its maximal ideal, the residue field of $\widehat{A}$ is still $k$. As $\widehat{A}$ is also henselian
, both maps $A\to k$ and $\widehat{A}\to k$ are henselian and so is $A\to \widehat{A}$ by cancellation. This implies that $A\to \widehat{A}$ is ind-etale iff it is an isomorphism. Now $A\to \widehat{A}$ is always formally smooth but not always an isomorphism.
\end{proof}

\begin{lemma}\label{fpet}
If $CRing^f$ is the category of maps of finite presentation, $indEt\cap CRing^f=Et^f$
\end{lemma}
\begin{proof}
Clearly $Et^f\subset indEt\cap CRing^f$, and as $indEt\subset fEt$ and $fEt\cap CRing^f=Et$, $indEt\cap CRing^f\subset Et^f$.
\end{proof}

The unique lifting system $(fEt,\overline{Nil})$ induces another unique factorisation system different from $(indEt,Hens)$ that we won't study here as it is not good (b. of definition \ref{good} fails).

\subsubsection{Points}

A ring $A$ corresponds to a point if it is non zero and any non zero map $A\to B\in Et^f$ admits a retraction.

\begin{prop}\label{etclosed}
A ring $A$ is a point for the $(Hens^o,indEt^o)$ system iff it is a nilpotent extension of a separably closed field.
\end{prop}
\begin{proof}
It is sufficient to prove that $A_{red}$ is a separably closed field. 
First, $A_{red}$ is a field from the fact that a localisation $A_{red}\to A_{red}[a^{-1}]$ is an etale map, so any non zero element of $A_{red}$ has to be invertible. Then a field is separably closed if, embedded in an algebraic closure, it contains all elements which minimal polynomial has simple roots. Any such polynomial $P$ being irreducible, it defines a normal extension $N$ of $A_{red}$ containing all roots of $P$; the map $A_{red}\to N$ is etale and the lifting property of $A$ gives a retraction, ensuring that all roots of $P$ were in $A_{red}$.

Conversely, if $A_{red}$ is a separably closed field, it is in particular en henselian local ring (\S\ref{hensel}). Now, for a henselian local ring $(B,m)$ with residue field $B/m=k$, an etale extension $B\to C$ has a retraction iff there exists a maximal ideal $n$ of $C$ sent to $m$ which residue field is also $k$ \cite[thm 4.2]{milne}. As $A_{red}=B=k$ in our case, for any $A_{red}\to C$ etale, a maximal ideal $n$ over $m$ always exist and as $k$ is separably closed the residue field at $n$ has to be $k$, so a retraction exists.
\end{proof}

\begin{prop}\label{etpoint}
The set of points of a ring $A$ is in bijection with that of prime ideals of $A$.
\end{prop}
\begin{proof}
Lemma \ref{locsep} implies $pt_{Et}(A)\subset pt_{Zar}(A)$.
The same reasoning as in proposition \ref{zarclosed} proves that separably closed fields are enough to compute points, and the inverse inclusion is then a consequence of the fact that any field has a separable closure.
\end{proof}

\subsubsection{Covering families and local objects}

\begin{prop}\label{etcov}
Point covering families of $Et$ are ordinary etale covers.
\end{prop}
\begin{proof}
By lemma \ref{fpet}, $Et^f=indEt\cap ({\cal C}^f)^o$. Then, by proposition \ref{etpoint}, a family of $A\to A_i$ of finitely presented etale maps is a cover iff it induces a surjective family on the set of prime ideals, which is the ordinary definition.
\end{proof}

A local ring $(A,m)$ is called {\em henselian} (\cite[thm. 4.2.d]{milne}) if any etale map $A\to B$ such that there exists a maximal ideal $n$ of $B$ lifting $m$ with the same residue field has a section.

\begin{prop}
A local ring $(A,m)$ is henselian iff $A\to A/m$ is an henselian map.
\end{prop}
\begin{proof}
Etale maps being stable by pushout, it is sufficient to prove the lifting property of $A\to A/m$ for squares
$$\xymatrix{
A\ar@{=}[r]\ar[d]&A\ar[d]\\
B\ar[r]&A/m
}$$
where $A\to B$ is etale.
As $A\to B$ is etale, $B\otimes_AA/m$ is separable extension of $k$, sum of the residue fields of maximal ideals of $B$ over $m$.
If $k$ is one of these fields, $k$ is an extension of $A/m$ and the map $B\to A/m$ gives a map $k\to B\otimes_AA/m\to A/m$ so in fact $k\simeq A/m$.
So any $A\to B$ entering such square is of the kind of extension used in the definition of a henselian ring. And conversely any such extension define a square like above.
\end{proof}

A henselian local ring $(A,m)$ is called {\em strictly henselian} if moreover $A/m$ is a separably closed field.

\begin{prop}
A ring $A$ is a pointed local object for the $(Hens^o,indEt^o)$ system iff it is a strictly henselian local ring.
\end{prop}
\begin{proof}
A point $K$ is a nilpotent extension of a separably closed field, so by lemma \ref{nilhens} $K\to K_{red}$ is a henselian map.
Therefore a map $A\to K$ is henselian iff $A\to K_{red}$ is (the necessary condition uses the cancellation property).
So a ring $A$ is pointed local iff there exists a henselian map $A\to K$ with $K$ a separably closed field.
As henselian maps are conservative, proposition \ref{zarpointablelocal} tells us that $A$ is a local ring. Then, if $m$ is the maximum ideal of $A$, $A\to K$ factors through $(A/m)^{sep}$, the separable closure of $A/m$ in $K$. Now, by construction, $(A/m)^{sep}\to K$ is henselian and the cancellation property says that so is $A\to (A/m)^{sep}$.
\end{proof}

\begin{prop}\label{etlocal}
A ring $A$ is a local object for the $(Hens^o,indEt^o)$ system iff it is a strictly henselian local ring.
\end{prop}
\begin{proof}
Local objects correspond to rings $A$ such that any etale cover $\{A\to A_i\}$ as a retraction of one of the $A\to A_i$.
As etale covers contain Zariski covers, $A$ is local by proposition \ref{zarlocal}. 

Now we are going to prove that $A\to k$ ($k$ residue field of $A$) is a henselian map.
Let $A\to B$ be an etale map lifting the residue field $k$, we need to show that it admits a section (necessary unique).
To prove this we consider an affine Zariski cover $\{A\to A_i,i\}$ of the complement of the closed point of $A$, the family $\{A\to B\}\cup\{A\to A_i;i\}$ is an etale cover (if fact even a Nisnevich cover, this will be useful to prove proposition \ref{hensel}). So there exists a map of this family admitting a retraction, and because all $A\to A_i$ are strict open embeddings it can only be $A\to B$.
It remains to prove that $k$ is separably closed. We are going to prove that any separable (\ie etale) extension $k\to k'$ admits a retraction.
$A$ being henselian there is a bijection between finite etale $A$-algebras and finite etale $k$-algebras, so $k'$ defines an etale $A$-algebra $A'$ which is an etale covering family (or can be completed as such in the same way as before), and so admit a retraction from $A$, proving the same for $k\to k'$.
\end{proof}

\subsubsection{Distinguished covering families}\label{fpcfet}

In order to apply theorem \ref{pointspec} we need to show that the condition of being a strict henselian ring can be tested using only finitely presented point covering families.

\medskip
A point covering family $\{B\to B_i,i\}$ of an $A$-algebra $B$ is said {\em distinguished} if every $B\to B_i$ is of finite presentation over $A$ and if it satisfies one of the following two conditions
\begin{enumerate}[a.]
\item it is a Zariski covering family,
\item or it consists of single etale map (such map will be called an {\em etale covering map}).
\end{enumerate}

\begin{lemma}\label{etalecov}
Any finitely presented etale map $B\to C$ between finitely presented $A$-algebras can be factored into a finitely presented localisation followed by a finitely presented etale covering map $B\to D\to C$.
\end{lemma}
\begin{proof}
The etale map $B\to C$ defines a degree function which associate to each point $p$ of $B$ the dimension of $C\otimes_B\kappa(p)$ as a $\kappa(p)$-vector space. This dimension is finite because the map is finitely presented and it is a lower semi-continuous function \cite[18.2.8]{EGA4-4}. The level set of value zero is a finitely presented closed Zariski subset whose complement is a localisation $D'$ of $B$. 
The natural map $C\to C\otimes_BD'$ is an isomorphism of $B$-algebras as it can be checked at every prime ideal of $B$, this gives a factorisation $B\to D'\to C$ of $B\to C$. 
We use the $(Loc,Cons)$ factorisation on $B\to C$ to obtain a localisation $D$ of $B$. As $B\to D'$ is another intermediate localisation, the universal property of $D$ gives a localisation $D'\to D$. Geometrically the Zariski spectrum of $D'$ contains that of $D$, which means that every prime ideal of $D$ has a non empty fiber over it. Conversely, if $K$ is a separably closed field and if $B\to C\to K$ is a point of $B$ factoring through $C$, it gives a map $B\to D\to K$ whose first map is a localisation, so $D$ has a map to the middle object $A_p$ of the $(Loc,Cons)$ factorisation of $B\to K$. This means that any prime ideal that has a non empty fiber is in $D$, and so $D=D'$. 
Finally, the map $D\to C$ is ind-etale and of finite presentation by cancellation.
\end{proof}

\begin{prop}\label{fpetale}
A $A$-algebra $B$ is a strictly henselian local ring iff it lifts through any distinguished covering families.
\end{prop}
\begin{proof}
The necessary condition is obvious by characterisation of local objects as strictly henselian rings.
Conversely, the lifting condition with respect to finitely presented Zariski covering families says that $B$ is a local ring (lemma \ref{prooflocal}). If $m$ is the maximal ideal of $B$ and $\kappa(m)^{sep}$ some separable closure of its residue field, we are going to prove that the map $B\to \kappa(m)^{sep}$ is henselian. It has to have the left lifting property with respect to finitely presented etale maps $C\to D$ between finitely presented $A$-algebras, we are going to transform this problem into a lifting through an etale covering map.
We can complete the lifting diagram as
$$\xymatrix{
&C[c^{-1}]\ar'[d]^{\textrm{et.cov.map}}[dd]\ar[rr]&&B[c^{-1}]\ar@{-->}[ld]_u^\simeq\\
C\ar[rr]\ar[dd]\ar[ru]&&B\ar[dd]\\
&D[c^{-1}]\ar[rd]\ar@{-->}[ru]^\ell\\
D\ar[rr]\ar[ru]&&\kappa(m)^{sep}
}$$
where $C\to C[c^{-1}]\to D$ is the localisation of lemma \ref{etalecov}.
The map $u$ exists and is an isomorphism because $B[c^{-1}]$ is a localisation of $B$ still containing the maximal ideal.
Now the lift $\ell$ exists by property of $B$.
\end{proof}

\subsubsection{Spectra and moduli interpretation}

Proposition \ref{etcov} ensures that the topology given by the general theory coincide with the usual etale topology for affine schemes.

\begin{prop}
For $A\in CRings^o$, $Spec_{Et}(A)$ is the usual etale spectrum (small etale topos) of $A$
and $SPEC_{Et}(A)$ is the usual big etale topos of $A$.
\end{prop}

As for the moduli interpretation of etale spectra, theorem \ref{pointspec} gives us something well known.

\begin{prop}\label{modulietale}
$SPEC_{Et}(A)$ classifies $A$-algebras that are strictly henselian local rings and 
$Spec_{Et}(A)$ classifies ind-etale $A$-algebras that are strictly henselian local rings.
In particular those $A$-algebras can have automorphisms and neither of $SPEC_{Et}(A)$ or $Spec_{Et}(A)$ is a spatial topos.
\end{prop}
\begin{proof}
The Nisnevich context $Et=(CRings^o=(indEt,Hens),\emptyset)$ is good: b. is true by construction and proposition \ref{fpetale} gives c. We use theorem \ref{pointspec}.
\end{proof}

Again in this case, the two notions of points agree.
\begin{prop}\label{sameptet}
For $A\in CRings$, the set of points of $Spec_{Et}(A)$ is in bijection with $pt_{Et}(A)$.
\end{prop}
\begin{proof}
We need to construct a bijection between the set of points of $Spec_{Et}(A)$ and the set of prime ideals of $A$.
First, for $p$ a prime ideal, we have the map $A\to A_{p}\to \kappa(p)\to\kappa(p)^{sep}$ where $\kappa(p)^{sep}$ is a separable closure of the residue field $\kappa(p)$. If $A\to A^{sh}_{p}\to \kappa(p)^{sep}$ is the $(indEt,Hens)$ factorisation of the previous map, $A^{sh}_{p}$ is a strictly henselian local ring (as a pointed local object) called a {\em strict henselisation of $A$ at $p$} (it depends up to a unique iso of the choice of $\kappa(p)^{sep}$). To recover $p$ from $A\to A^{sh}_{p}$ we are going to show that the composition $A\to A_{p}\to A^{sh}_{p}$ is the $(Loc,Cons)$ factorisation of $A\to A^{sh}_{p}$, so $A_{p}$ will be uniquely determined by $A^{sh}_{p}$. We only need to prove that $h:A_{p}\to A^{sh}_{p}$ is conservative: in the square
$$\xymatrix{
A_{p}\ar[r]^{Cons}\ar[d]_h&\kappa(p)\ar[d]^{\iota}\\
A^{sh}_{p}\ar[r]_{Cons}&\kappa(p)^{sep}
}$$
the top and bottom maps and $\iota$ are conservative then $h$ is conservative by cancellation.
All this creates an injective map from the set of prime ideals of $A$ to that of points of $Spec_{Et}(A)$.
We prove now that this map is surjective. If $B$ is a stricly henselian local ring with residue field $K$ separably closed, and $A\to B$ an ind-etale map, the $(Loc,Cons)$-factorisation of $A\to B$ give a local ring $A_{p}$. The map $A_{p}\to K$ factors through some separable closure of $\kappa(p)$ and defines a strict henselisation $A^{sh}_{p}$ of $A$ at $p$.
With the above notations, we have the diagram
$$\xymatrix{
A\ar[r]^-{indEt}\ar[d]_{Loc}&B\ar[r]^-{Hens}&K\\
A_{p}\ar[r]_-{indEt}\ar[ru]_{Cons}&A^{sh}_{p}\ar[r]_-{Hens}&\kappa(p)^{sep}\ar[u]_{Hens}.
}$$
Then the composite map $A\to B\to K$ admits another $(indEt,Hens)$ factorisation $A\to A^{sh}_{p}\to K$ so $B\simeq A'_{p}$.
\end{proof}

This proof gives the following construction of the ind-etale henselian local $A$-algebra at a prime $p\subset A$: it is the middle object $A^{sh}_{p}$ of the $(indEt,Hens)$ factorisation of the map $A\to \kappa(p)^{sep}$ where $\kappa(p)^{sep}$ is a separable closure of the residue field at $p$.

\subsubsection{Remark}\label{remzaret}

The two factorisation systems $(Loc,Cons)$ and $(indEt,Hens)$ are related by the inclusion $Loc\subset indEt$. For a map $A\to B$, this constructs in fact a triple factorisation system
$$\xymatrix{
A\ar[r]^{Loc} &C\ar[rr]^{indEt\ \&\ Cons}&& D\ar[r]^{Hens}&B
}$$
where $A\to C\to B$ is the $(Loc,Cons)$ factorisation and $A\to D\to B$ the $(indEt,Hens)$ factorisation.
As shown in lemma \ref{etalecov}, the map $A\to C$ is the "open support" of the etale map $A\to D$ and the map $C\tto D$ is an etale covering.

This triple factorisation will be inspire the construction of the $(Int,IntClo)$ factorisation system in \S\ref{finitetop}.

\subsection{Nisnevich topology}\label{nisnevich}

The Nisnevich topology on $CRings^o$ is not associated to a factorisation system, but will be constructed from the etale factorisation system by Nisnevich forcing (\S\ref{forcing}), more precisely by forcing fields, to be local objects.
The setting is the same as in \S\ref{etale}.

\medskip
An etale point covering family $A\to A_i$ is a {\em Nisnevich covering family} if for any field $K$ and any map $A\to K$ has a factorisation through some $A_i$:
$$\xymatrix{
A_i\ar@{-->}[dr]^{\exists}&\\
A\ar[u]\ar[r]&K.
}$$
If $\cal F$ is the subcategory of $CRings$ generated by fields, $(CRings^o=(Hens^o,indEt^o),{\cal F})$ is a Nisnevich context.

\medskip
The following lemma is a consequence of lemma \ref{locsep} and of the definition of Nisnevich covering families.
\begin{lemma}\label{zarnis}
Zariski point covering families are Nisnevich covering families.
\end{lemma}

\subsubsection{Local objects}

\begin{prop}\label{hensel}
A ring is a Nisnevich local object iff it is a henselian local ring.
\end{prop}
\begin{proof}
Let $A$ be a local object. As Zariski covering families are Nisnevich covering families, proposition \ref{zarlocal} shows that $A$ is a local ring.
Let $k$ be the residue field of $A$, we need to prove that $A\to k$ is a henselian map. The argument is the one use in the proof of proposition \ref{etlocal}.
\end{proof}

Let $f{\cal F}$ be the category of fat fields, \ie nilpotent extension of fields (\S\ref{zarpoints}).

\begin{cor}\label{gennisnevich}
$(CRings^o=(Hens^o,indEt^o),{\cal F})$ and $(CRings^o=(Hens^o,indEt^o),f{\cal F})$ are two equivalent Nisnevich contexts.
\end{cor}
\begin{proof}
As ${\cal F}\subset f{\cal F}$, localising by $f{\cal F}$ selects less covering families so more local objects: $\overline{\cal F}\subset \overline{f{\cal F}}$. The reciprocal inclusion is equivalent to fat fields being henselian rings, \ie that the map $K\to K_{red}$ is henselian. This is a consequence of lemma \ref{nilhens}.
\end{proof}

This corollary is interesting as $f{\cal F}$ is exactly the category of points of the $(Loc,Cons)$ factorisation system (\S\ref{zarpoints}), which is a way to say that this Nisnevich localisation is not arbitrary (see \S\ref{duality}).

\subsubsection{Distinguished covering families}\label{fpcfnis}

A Nisnevich point covering family $\{B\to B_i,i\}$ of an $A$-algebra $B$ is said {\em distinguished} if it is of finite presentation over $A$ (\ie there exist $A\to B'\to B$ where $A\to B'$ is of finite presentation and all $B\to B_i$ are pushout of some $B'\to B'_i$) and satisfy one of the following two conditions
\begin{enumerate}[a.]
\item it is a Zariski covering family, 
\item or there exist a radical ideal $I$ of $B'$ such that $A\to A/I$ factors through one of the $B'\to B'_i$ and the others $B\to B_i$ are localisations of $B'$ covering the complement of $I$. In particular, this implies that the $B'_i$ factoring $B'\to B'/I$ is unique.
\end{enumerate}
Geometrically (for the Zariski topology), this last condition says that the covering family is distinguished if it covers the complement of a finitely presented closed set $Z$ by Zariski opens and has another etale map covering $Z$ that moreover has a section over $Z$. This was inspired by Nisnevich distinguished squares.

\begin{prop}
A $A$-algebra $B$ is a henselian local ring iff it lifts through any distinguished Nisnevich covering families.
\end{prop}
\begin{proof}
We need to prove only the sufficient part. Lifting through finitely presented Zariski covering families says that $B$ is a local ring (lemma \ref{prooflocal}), we need then to show that, if $m$ is the maximal ideal of $B$ and $\kappa(m)$ its residue field, the map $B\to \kappa(m)$ is henselian. This is true if it has the left lifting property with respect finitely presented etale maps $C\to D$ between finitely presented $A$-algebras, we can use the same trick as in proposition \ref{fpetale} and replace $C\to D$ by an etale covering map.
We are now going to transform $C\to D$ into a distinguished Nisnevich covering of the second kind. 
The Zariski closed set involved will be the closure $\overline{p}$ of the image $p$ of the ideal $m$ by $C\to B$, but we need to show that $C\to D$ has a section over it. The finitely presented etale map $\kappa(p)\to D\otimes_C\kappa(p)$ has a section which furnishes an idempotent of $D\otimes_C\kappa(p)$ \cite[cor. 3.12]{milne}, this idempotent can be lifted as some element $d\in D$ and the composition $C\to D[d^{-1}]$ is still finitely presented etale covering map but is now of degree exactly one over $p$. The set $Z$ of prime ideals of $C$ for which $C\to D$ is of degree exactly 1 is a closed Zariski subset, over which $C\to D$ is even an isomorphism. Then, the wanted section exists as $Z$ contains $\overline{p}$.
Completing $C\to D$ by a Zariski covering of the complement of $\overline{p}$, and pushing forward to $B$, there exists a retraction of one of the covering maps and it can be only of $B\to D\otimes_CB$ as all other maps misses $m$ in their image by construction.
\end{proof}

\subsubsection{Spectra and moduli interpretation}\label{nispec}

Let $\cal F$ be the full subcategory of ${\cal C}=CRings^o$ generated by fields. $Nis:=({\cal C}=(Hens^o,indEt^o),{\cal F})$ is a Nisnevich context (def.\ref{niscontext}) and proposition \ref{hensel} says that $\overline{\cal F}$ is the category of henselian rings.

\begin{prop}
For $A\in CRings^o$, 
$SPEC_{Nis}(A)$ classifies $A$-algebras that are henselian local rings and
$Spec_{Nis}(A)$ classifies ind-etale $A$-algebras that are henselian local rings.
In particular those $A$-algebras can have automorphisms and neither of $SPEC_{Nis}(A)$ or $Spec_{Nis}(A)$ is a spatial topos.
\end{prop}
\begin{proof}
The Nisnevich context $Nis$ is good: b. is true by construction of the factorisation system and distinguished families were constructed in \S\ref{fpcfnis}. We apply theorem \ref{pointspec}.
\end{proof}

The $Spec_{Nis}(A)$ contains $Spec_{Et}(A)$ as a strict subtopos, the set of points of $Spec_{Nis}(A)$ is then bigger than that of prime ideal of $A$.

To any prime ideal $p$ of $A$ is associated two points of $Spec_{Nis}(A)$: first, $Spec_{Et}(A)$ being a subtopos of $Spec_{Nis}(A)$, the strict henselisation of $A$ at $p$ is also a point of $Spec_{Nis}(A)$ ; the second one is the {\em henselisation} of $A$ at $p$: it is the middle object $A^h_{p}$ of the $(indEt,Hens)$ factorisation of the map $A\to \kappa(p)$ where $\kappa(p)$ is the residue field at $p$.

\subsubsection{Context Comparisons}

We have three Nisnevich contexts: $Zar$, $Et$ and $Nis$.
$Et$ is clearly a refinement of $Nis$ and of $Zar$ and as objects of ${\cal F}$ are local for $Zar$, $Nis$ is also a refinement of $Zar$.
This gives the following diagram of spectra:
$$\xymatrix{
Spec_{Et}(X)\ar[d]^{s_X}\ar[r]&Spec_{Nis}(X)\ar[d]^{s_X}\ar[r]&Spec_{Zar}(X)\ar[d]^{s_X}\ar[r]&Spec_{Ind}(X)\ar[d]^{s_X}\\
SPEC_{Et}(X)\ar[r]&SPEC_{Nis}(X)\ar[r]&SPEC_{Zar}(X)\ar[r]&SPEC_{Ind}(X)
}$$
and associated natural transformations of structural sheaves ${\cal O}_X^{Et}\to {\cal O}_X^{Nis}\to {\cal O}_X^{Zar}\to X$ (pulled-back on $Spec_{Et}(X)$. 
The bottom row of the diagram consists in inclusions of subtoposes, and reads at the level of points: strict henselian local rings are henselian local rings which are local rings which are rings.
The top row is formed of the moduli of different notions of local forms of $X$.

\paragraph{The example of a field}
The Etale topos is the classifying topos of the Galois group of $k$, its category of points is the groupoid of separable closure of $k$. The Zariski topos of a field $k$ is a point, but the Nisnevich topos of a field is not, its category of points is the opposite of that of algebraic extensions of $k$.
(As $k$ is henselian, an ind-etale $k$-algebra $A$ is a product of local $k$-algebras $A_i$ and if $A$ is henselian so are the $A_i$. If $k_i$ is the residue field of $A_i$, as both maps $k\to k_i$ and $k\to A_i$ are ind-etale so is $A_i\to k_i$. It is then an isomorphism if $A_i$ is henselian.)
This category has an terminal object ($k$ itself) and geometrically, the Nisnevich spectrum can be thought as a sort of cone interpolating between $k$ and the groupoid of its separable closures. This picture shows also why it is homotopically contractible.

\subsection{Domain topology}

We are now going to investigate the obvious $(Surj,Mono)$ factorisation system on $CRings$ with the same convention as before, \ie thinking of the opposite factorisation system $(Mono^o,Surj^o)$ on $CRings^o$.
Let $u:A\to B\in CRings$ with kernel $I$, the $(Surj,Mono)$ factorisation of $u$ is $A\to A/I\to B$. 
A map $A\to A/I$ is called a {\em surjection} or a {\em quotient} and a map $A\to B$ with 0 kernel is called a {\em monomorphism}.

The following lemma gives a set of left generators.
\begin{lemma}\label{leftgendom}
A map is a monomorphism iff it has the right lifting property with respect to $\ZZ[x]\tto \ZZ:x\mapsto 0$.
\end{lemma}

It is interesting to remark that this map $\ZZ[x]\tto \ZZ$ is the "complement" of the generator $\ZZ[x]\tto \ZZ[x,x^{-1}]$ of the $(Loc,Cons)$ system. This simple fact seems to be the source of an unclear duality between the $(Surj,Mono)$ and $(Loc,Cons)$ systems (\S\ref{duality}).

\subsubsection{Points}\label{fcsurj}

\begin{prop}\label{pointsurjmono}
A ring corresponds to a points of the $(Mono^o,Surj^o)$ factorisation system iff it is a field.
\end{prop}
\begin{proof}
A ring $A$ corresponds to a point if any quotient $A\to A/I$ by a finitely presented ideal admits a retraction.
But this forces $q$ to be a monomorphism and then an isomorphism. 
An element $a\in A$ is either zero, invertible or non-zero and non invertible. In the first case the quotient by $a$ is $A$, in the second 0 and in the third something non isomorphic to $A$. This third case is excluded by the previous remark, so every element in $A$ as to be either zero or invertible.
\end{proof}

The same classical argument as in proposition \ref{zarpt} gives the following.
\begin{prop}
The set of point of a ring $A$ for the $(Mono^o,Surj^o)$ system is that of prime ideals of $A$.
\end{prop}

\subsubsection{Covering families and local objects}

Point covering families of the $(Mono^o,Surj^o)$ system are families of quotients $A\to A/I_i$ by finitely generated ideals such that any residue field of $A$ factors through one of the $A/I$.
Using the geometric intuition coming from the Zariski topology, this correspond to cover a scheme by non reduced closed subschemes of finite codimension.

\medskip
A ring $B$ is an integral domain iff for any $x,y\in B$, $xy=0$ iff $x=0$ or $y=0$.
If $B$ is an $A$-algebra this can be read as, for any $x,y\in B$, the map $A[x,y]/(xy)\to B$ factors through $A[x,y]/(xy)\to A[y]$ or $A[x,y]/(xy)\to A[x]$. Those two maps form a covering family of $A[x,y]/(xy)$: for any map $A[x,y]/(xy)\to K$ to some field, either $x$ or $y$ has to be zero in $K$.

This proves the following lemma dual to lemma \ref{prooflocal}.
\begin{lemma}\label{proofintegral}
A $A$-algebra $B$ is a integral domain iff for any $x,y\in B$ such that $xy=0$ is invertible, $B$ lift through the point covering family 
$A[x,y]/(xy)\to A[x]$ and $A[x,y]/(xy)\to A[y]$ of $A[x,y]/(xy)$.
\end{lemma}

The following results justify the name chosen for this topology.

\begin{prop}
A ring is a pointed local object of the $(Mono^o,Surj^o)$ system iff it is an integral domain.
\end{prop}
\begin{proof}
If $A\to K$ is a monomorphism with target a field, then $A$ is an integral domain, and conversely for any such ring is associated a monomorphism $A\to K(A)$ into the fraction field.
\end{proof}

\begin{prop}\label{localsurj}\label{proofdomain}
A ring is a local object of the $(Mono^o,Surj^o)$ system iff it is an integral domain.
\end{prop}
\begin{proof}
Let $A$ be a domain and $\{A\to A/I_i\}$ a cover, then in order to cover the generic point of $A$ it must contain a copy of $A$ itself.
Conversely, if $A$ is a ring such that any cover $\{A\to A/I_i\}$ has a retraction, the family of inclusions of irreducible components, \ie $A\to A/p_i$ where $p_i$'s are minimal prime ideals, defines a point covering family of $A$ and then must have a retraction. So 0 is one (and the only) of the primes $p_i$.
\end{proof}

\subsubsection{Spectra and moduli interpretation}

\begin{prop}
For $A\in CRings$, 
points of $SPEC_{Dom}(A)$ are $A$-algebras that are integral domains and
points of $Spec_{Dom}(A)$ are quotients of $A$ that are domains.
\end{prop}
\begin{proof}
The Nisnevich context $Dom=(CRings^o=(Mono^o,Surj^o),\emptyset)$ is good: b. is lemma \ref{leftgendom} and distinguished families are defined in lemma \ref{proofintegral}. We apply theorem \ref{pointspec}.
\end{proof}

\begin{prop}
For $A\in CRings$, the set of points of $Spec_{Dom}(A)$ is in bijection with $pt_{Dom}(A)$.
\end{prop}
\begin{proof}
This a way to say that the set of points of $Spec_{Dom}(A)$ is that of prime ideals of $A$: it is well known that a quotient of $A$ that is an integral domain iff the kernel is a prime ideal.
\end{proof}

We are now going to prove that $Spec_{Dom}(A)$ is a topological space.
We would like to apply the same argument as in proposition \ref{zarspace} but the equivalence between the jointly surjective topology and the point covering topology will fail without a slight modification of the site defining $Spec_{Dom}(A)$ (end of proof of proposition \ref{surjspace}).

\medskip
For a ring $A$, the subset $\sqrt{0}$ of all nilpotent elements is also the intersection of all prime ideals of $A$.
\begin{lemma}\label{equivsite}
For a ring $A$, $A\to A/\sqrt{0}$ is point covering family and any sheaf for the factorisation topology send such a map to an isomorphism.
\end{lemma}
\begin{proof}
Any field $K$ is an integral domain so any $A\to K$ factors through $A/\sqrt{0}$.
For the second part, any sheaf as to send $A\to A/\sqrt{0}$ to the kernel of $A/\sqrt{0}\rightrightarrows A/\sqrt{0}\otimes_AA/\sqrt{0}$, but 
as $A/\sqrt{0}\otimes_AA/\sqrt{0}=A/\sqrt{0}$ this kernel is the identity of $A/\sqrt{0}$.
\end{proof}

\begin{cor}
The domain topology is not subcanonical.
\end{cor}
\begin{proof}
Both $A$ and $A_{red}$ will have the same spectra, this will be developped further below.
\end{proof}

\begin{lemma}\label{sqrt}
A family $B\to B/I_i$ in $A/Surj^f$ corresponds to a point covering family iff $B\to B/(\cap I_i)$ is a point covering family iff $\cap I_i\subset\sqrt{0}$.
\end{lemma}
\begin{proof}
$B\to B/(\cap I_i)$ factors every $B\to B/I_i$, so it has the joint of the lifting properties of all $B\to B/I_i$ and so is a point covering family.
Conversely, if $I_i=(a_i^1,\dots a^{k_i}_i)$, $\cap I_i$ is generated by products $\prod_i a^{k(i)}_i$ for some function $i\mapsto 1\leq k(i)\leq k_i$, we want to prove that for any point $A\to K$ factoring through $A\to A/(\cap I_i)$, there exists an $i$ such that all $a^k_i$ are send to zero in $K$. If this is not the case, for all $i$ there would exist a $a^{k(i)}_i$ not sent to zero in $K$, and so their product will not either, contradicting the fact that $A\to K$ factors through $A\to A/(\cap I_i)$.

As for the second equivalence, if $p$ is a prime ideal of $B$ with residue field $\kappa(p)$, the existence of a lift $B/(\cap I_i)\to \kappa(p)$ of $B\to \kappa(p)$ proves that $p$ has not become the zero ideal in $B/(\cap I_i)$ so $\cap I_i\subset p$. This says that $(\cap I_i)$ is contained in every prime ideal of $B$.
\end{proof}

$Spec_{Dom}(A)$ is the topos associated to $(A/Surj^f)^o$ with the factorisation topology, it depends only of $A_{red}$.
If $A/RedSurj^f$ is the sub-category of $A/Surj^f$ formed of reduced finitely presented quotients of $A$, the factorisation topology restrict to it. The inclusion $\iota:A/RedSurj^f\subset A/Surj^f$ has a left adjoint $red$ given by $A\to A_{red}=A/\sqrt{0}$ which if continuous (the reduction of a covering family is still a covering family).

\begin{lemma}\label{equivcov}
A family $B\to B/\sqrt{I_i}$ in $(A/RedSurj^f)^o$ is point covering family iff $\cap \sqrt{I_i}=\sqrt{0}$.
\end{lemma}
\begin{proof}
This is a consequence of lemma \ref{sqrt} and of $\cap \sqrt{I_i}=\sqrt{\cap I_i}$.
\end{proof}

\begin{prop}
The continuous functor $red:(A/Surj^f)^o\tto (A/RedSurj^f)^o$ is an equivalence of sites.
\end{prop}
\begin{proof}
Recall that a continuous functor is an equivalence of sites if the geometric map $(red^*,red_*)$ induced on the toposes is a equivalence.
We have a diagram
$$\xymatrix{
\widehat{(A/RedSurj^f)^o}\ar@<-.4ex>[d]_a\ar@<.4ex>[r]^-{red^*}&\widehat{(A/Surj^f)^o}\ar@<.4ex>[l]^-{red_*}\ar@<-.4ex>[d]_a\\
\widetilde{(A/RedSurj^f)^o}\ar@<-.4ex>[u]\ar@<.4ex>[r]^-{red^*}&\widetilde{(A/Surj^f)^o}\ar@<-.4ex>[u]\ar@<.4ex>[l]^-{red_*}
}$$
where the $a$'s are the sheafification functors.
We have to prove that a presheaf on $(A/Surj^f)^o$ is a sheaf iff its restriction to $(A/RedSurj^f)^o$ is a sheaf. It is enough to check it on the level of generators where $red_*=\iota^*$. The unit and counit of $(red^*,red_*)$ are those of $(red,\iota)$: the counit is always an isomorphism and lemma \ref{equivsite} prove that the unit of $(red,\iota)$ is transform in an isomorphism by sheafification.
\end{proof}

\begin{prop}\label{surjspace}
$Spec_{Dom}(A)$ is a topological space whose poset of points is equivalent to that of prime ideal of $A$.
\end{prop}
\begin{proof}
We are going to apply the same argument as in proposition \ref{zarspace}.
$Spec_{Dom}(A)$ is generated by the category $(A/RedSurj^f)^o$ which is a poset of compact object, so it is a localic topos. \cite[II.3.]{stone} will say it is coherent and spatial as soon as the topology on $(A/Surj^f)^o$ is the jointly surjective topology.
$(A/Surj^f)^o$ is a distributive lattice: the intersection of $A/\sqrt{I}$ and $A/\sqrt{J}$ is $A/\sqrt{I+J}$ and the union is $A/\sqrt{I\cap J}$; the distributivity law is the lemma: for $I,J,K$ three finitely generated ideals of $A$, $K+(I\cap J)=K\cap I+K\cap J$.
As for the topology, a family $A\to A/\sqrt{I_i}$ is jointly surjective iff $\sqrt{\cap I_i}=\sqrt{0}$ but this is the characterisation of point covering families of lemma \ref{equivcov}. (This last equivalence is in fact the whole reason of considering the site $(A/RedSurj^f)^o$.)
\end{proof}

The poset of points of $Spec_{Dom}(A)$ is the opposite of that of $Spec_{Zar}(A)$, in particular the generic points of one are closed points of the other. In fact the two sites $A/RedSurj^f$ and $A/Loc^f$ are opposite categories and this duality between $Spec_{Dom}(A)$ and $Spec_{Zar}(A)$ is part of a general duality on compactly generated spaces exposed in \cite{stone} ($Spec_{Dom}(A)$ is the domain spectrum of \cite[V.3.11]{stone}).

\subsubsection{Remark}\label{variation2}

The same remark as in \S\ref{variation} is true: the class $Surj^o$ is not local for the Domain topology on $CRings^o$. Its saturation is the class $EtSurj^o$ opposite to that of integrally closed maps (\S\ref{finitetop}) that are locally trivial for the Domain topology (called {\em etale-surjective } maps). Again, we claim that $EtSurj$ is the left class of a unique factorisation system $(EtSurj,MIdem)$ on $CRings$ where $MIdem$ is the class of monomorphisms having the extra unique lifting property for idempotents, \ie $(EtSurj,MIdem)$ is left generated by $\ZZ[x]\to\ZZ$ and $\ZZ\to\ZZ\times\ZZ$. 
Replacing the factorisation system $(Loc,Cons)$ by $(EtSurj,MIdem)$ in the previous study would generate the same factorisation topology and the same spectra.

\subsection{Finite topology}\label{finitetop}

For a inclusion of rings $A\subset B$ an element $b\in B$ is said {\em integral over $A$} if there exists a monic polynomial $P$ with coefficients in $A$ such that $b$ is a root of $P$. In particular every element of $A$ is integral.
More generally for any map $A\to B$ of kernel $I$, an element of $B$ is said integral over $A$ if it is integral over $A/I$. As any monic polynomial of $(A/I)[X]$ can be lifted in a monic polynomial of $A[X]$, it is equivalent to say that $b\in B$ is integral over $A$ if it exists $P\in A[X]$ monic such that $P(b)=0$.
$A\subset B$ is said {\em integrally closed} if any element integral over $A$ is in $A$.
The set of integrally closed monomorphism of rings is noted $IntClo$.
The following proposition is proposition 5.1 and corollary 5.3 of \cite{atiyahmacdonald}.
\begin{prop}\label{amd}
For any monomorphism of rings $A\subset B$
\begin{enumerate}[a.]
\item an element $b\in B$ is integral over $A$ iff the sub-$A$-algebra of $B$ generated by $b$ is finitely generated as an $A$-module ;
\item the subset $C$ of elements integral over $A$ in $B$ is a ring, and $C\subset B$ is integrally closed.
\end{enumerate}
\end{prop}

This constructs a factorisation system on monomorphisms of rings, with the right class being $IntClo$.
To have a factorisation for every morphism, we use the $(Surj,Mono)$ factorisation.
A map $A\to B$ of is called {\em integral}, or an {\em integral extension} if every element of $B$ is integral over $A$.
The set of integral maps is noted $Int$.
The archetype of a integral map is a integral extension $A\to (A/I)[x]/P(x)$ for some ideal $I$ and some monic polynomial $P$.

\begin{prop}
$Int$ and $IntClo$ are the left and right classes of a unique factorisation system.
\end{prop}

As $IntClo\subset Mono$ and $Surj \subset Int$, the $(Int,IntClo)$ factorisation system compares to the $(Surj,Mono)$ as $(Loc,Cons)$ and $(indEt,Hens)$ compared in \S\ref{remzaret}: they define a triple factorisation system
$$\xymatrix{
A\ar[r]^{Surj} &C\ar[rr]^{Mono\ \&\ Int}&& D\ar[r]^{IntClo}&B
}$$
where $A\to C\to B$ is the $(Surj,Mono)$ factorisation and $A\to D\to B$ the $(Int,IntClo)$ factorisation.

\begin{prop}\label{intgen}
The $(Int,IntClo)$ factorisation system is left generated by the set of maps $A\to(A/I)[x]/P(x)$ where $A$ is of finite presentation, $I$ is some finitely generated ideal of $A$ and $P$ is a monic polynomial.
\end{prop}
\begin{proof}
First, it is clear by definition that such a map $A\to (A/I)[x]/P(x)$ is in $Int$. 
Then, as a factorisation system is entirely determine by one of the left or right classes, it is sufficient to prove that the class of maps right orthogonal to $A\to (A/I)[x]/P(x)$ is $IntClo$.
For a map $B\to C$, a lifting for the square 
$$\xymatrix{
\ZZ[x]\ar[d]\ar[r]&B\ar[d]\\
\ZZ\ar[r]\ar@{-->}[ru]&C,
}$$
exists iff the kernel of $B$ is reduced to 0, \ie that $B\to C$ is a monomorphism.
Now for a square (with $P$ monic and $B\to C$ a monomorphism)
$$\xymatrix{
A\ar[d]\ar[r]&B\ar[d]^{mono}\\
A[x]/P(x)\ar[r]\ar@{-->}[ru]&C
}$$
the image of $x$ in $C$ is an element integral over $B$ and any such can be defined by such a square.
The existence of a lift states that any element integral over $B$ is image of an element in $B$, \ie that $B$ is integrally closed in $C$.
\end{proof}

\begin{lemma}\label{fingenlem}
A finitely presented map $A\to B$ is an integral map iff it is finite.
\end{lemma}
\begin{proof}
$A\to B$ is a finite map if $B$ is finitely presented as a $B$ module. If the map is moreover finitely presented in $CRing$, $B$ is finitely presented as an $A$-module.

An integral extension is finite iff it is finitely generated. Conversely, if $A\to B$ is a finitely presented finite map, we prove that every element $b\in B$ is integral over $A$: the sub-$A$-algebra $C$ generated by $b$ is finitely presented $A$-module as a sub-module of such, then we use proposition \ref{amd}.
\end{proof}

Let $Fin^f$ be the class of finitely presented finite maps in $CRings$ and $Fin^f_*$ be the class of finite maps between finitely presented rings. The following proposition justifies the name chosen for this topology.
\begin{prop}\label{fingen}
The $(Int,IntClo)$ factorisation system is left generated by $Fin^f$. For this choice of generator, the class $Int$ is $Ind\z Fin^f$. 
\end{prop}
\begin{proof}
$Fin^f\subset Int$ by lemma \ref{fingenlem} and it contains the generators of proposition \ref{intgen}.
Using notation of the proof of theorem \ref{factolim}, the second assertion comes from $\overline{G}=Fin^f$.
\end{proof}

\subsubsection{Points}\label{fcintsurj}

\begin{prop}
A ring corresponds to a point of the $(Int,IntClo)$ factorisation system iff it is an algebraically closed field.
\end{prop}
\begin{proof}
A ring $A$ corresponds to a point iff any finitely presented integral map $A\to B$ admits a section.
From proposition \ref{intgen}, it is necessary and sufficient to prove this only for maps $A\to B$ where $B=A/I$ for some finitely generated ideal $I$ or $B=A[x]/P(x)$ and some monic or zero polynomial $P$.
Proposition \ref{pointsurjmono} says that existence of retraction for quotients $A\to A/I$ implies that $A$ is a field.
A field $A$ is now a point iff every monic polynomial has a root in $A$. But with coefficients in a field every polynomial is proportional to a monic one and $A$ is a point iff every polynomial has a root in $A$.
\end{proof}

\begin{prop}
The set of points of a ring $A$ is in bijection with the set of prime ideals of $A$.
\end{prop}
\begin{proof}
As $Surj\subset Int$, $pt_{Prop}(A)\subset pt_{Dom}(A)$. The inverse inclusion is a consequence of the existence of an algebraic closure for every field.
\end{proof}

\subsubsection{Covering families and local objects}\label{strictintegrallycloseddomain}

A family $\{A\to A_i\}$ of finitely presented integral maps is a point covering family iff any map $A\to \overline{k}$ to a residual algebraically closed field factors through some $A\to A_i$.
This is equivalent to the fact that any map $A\to k$ to a residue field of $A$ lift through one of the $A\to A_i$ after an algebraic extension of $k$.

\begin{prop}
Pointed local objects are integrally closed domain which fraction field is algebraically closed.
\end{prop}
\begin{proof}
Let $K$ be an algebraically closed field, and $A\to K$ an integrally closed map. We need only to show that the fraction field $K(A)=A[(A^*)^{-1}]$ of $A$ is algebraically closed. But the stability by localisation of integral closure implies that $K(A)\to K[(A^*)^{-1}]\simeq K$ is again integrally closed.
\end{proof}

In analogy with strict henselian local rings, such rings will be called {\em strict integrally closed domains}.

\begin{prop}\label{localintsurj}
Local objects are integrally closed domain which fraction field is algebraically closed.
\end{prop}
\begin{proof}
Let $A$ be a local object. As it must be a local object for the $(Surj,Loc)$ factorisation system, it must be an integral domain.
Now we have to prove that the map $A\to K(A)=A[(A^*)^{-1}]$ is integrally closed. As is it already a monomorphism it is sufficient to prove that it has the unique right lifting property with respect to maps $A\to A[x]/P(x)$ where $P$ is monic. We are going to use the same argument as for the local objects of etale topology.
Given such a map $A\to A[x]/P(x)$ lifting the fraction field of $A$, it can be completed in a $(Int,IntClo)$ point covering family by adjoining $A\to (A/p)^{int}$ for prime ideals different from 0. Then the hypothesis on $A$ gives a retraction of one the map of the cover which can only be $A\to A[x]/P(x)$.

To prove that the fraction field $K(A)$ is algebraically closed, we are going to prove that any algebraic extension $K(A)\to K(A)[x]/P(x)$ where $P$ is irreducible in $K(A)[X]$ has a retraction.
The composite $A\to K(A)\to K(A)[x]/P(x)$ factors as $A\to A'\to K$ where $A'$ is the integral closure of $A$ in $K$, this map $A\to A'$ is a $(Int,IntClo)$ point covering family (or can be completed as such in the same way as before) and thus admits a retraction, which gives the wanted retraction for $K(A)$.
\end{proof}

\subsubsection{Distinguished covering families}\label{fpcfintsurj}

A point covering family $\{B\to B_i,i\}$ of an $A$-algebra $B$ is said {\em distinguished} if all the $B\to B_i$ are maps of finite presentation of $A$-algebras and satisfy one of the following two conditions
\begin{enumerate}[a.]
\item it is a $(Mono^o,Surj^o)$ point covering family,
\item or it consists of single integral extension (such map will be called an {\em integral covering map}).
\end{enumerate}

\begin{lemma}
Any finitely presented integral map $B\to C$ between finitely presented $A$-algebras can be factored into a finitely presented quotient followed by a finitely presented integral covering map.
\end{lemma}
\begin{proof}
We use the $(Surj,Mono)$ factorisation on $B\to C$ to obtain a quotient $D/I$ of $B$ with $I$ the kernel of $B\to C$. $I$ is finitely generated so $D\to B$ is finitely presented and so is $D\to C$ by cancellation.

We have to prove that $D\to C$ is an integral covering map. $C$ is generated by some finite set $\{c_i;i\}$ where $c_i$ is a root of a monic polynomial in $B[c_1,\dots, c_{i-1}][x]$. 
If $K$ is algebraically closed and $D\to K$ is a point, $K\to C\otimes_DK$ is an algebraic extension generated by the image of the $c_i$ (because the relations are monic, $C\otimes_DK$ is not zero). As $K$ is algebraically closed there exists a retraction, proving that any point of $D$ lift though $D\to C$.
\end{proof}

\begin{prop}\label{fpintsurj}
An $A$-algebra $B$ is a strictly integrally closed domain iff it lifts through any distinguished covering families.
\end{prop}
\begin{proof}
The necessary condition is clear by characterisation of local objects as strict integrally closed rings.
Conversely, the lifting condition with respect to finitely presented $(Mono^o,Surj^o)$ point covering families says that $B$ is a integral domain (lemma \ref{proofdomain}). 

If $K(B)^{alg}$ is an algebraic closure of the fraction field of $B$, we are going to prove that $B\to K(B)^{alg}$ is integrally closed.
It needs to have the left lifting property with respect to finitely presented integral map $C\to D$ between finitely presented $A$-algebras, we can transform this problem into a lifting through an integral covering map.
$$
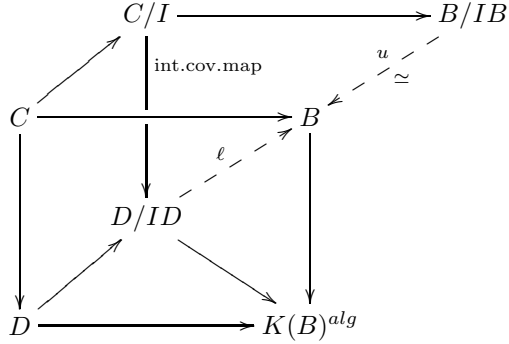
\xymatrix{
&C/I\ar'[d]^{\textrm{int.cov.map}}[dd]\ar[rr]&&B/IB\ar@{-->}[ld]_u^\simeq\\
C\ar[rr]\ar[dd]\ar[ru]&&B\ar[dd]\\
&D/ID\ar[rd]\ar@{-->}[ru]^\ell\\
D\ar[rr]\ar[ru]&&K(B)^{alg}
}$$
where $I$ is the kernel of $C\to D$. The map $u$ exists and is an isomorphism as $B/IB$ is a quotient of $B$ still containing the generic point.
And the lift $\ell$ exists by property of $B$.
\end{proof}

\subsubsection{Spectra and moduli interpretation}

\begin{prop}
$SPEC_{Prop}(A)$ classifies $A$-algebras that are strict integrally closed domains
and $Spec_{Prop}(A)$ classifies integral $A$-algebras that are strict integrally closed domains. 
In particular those algebras can have automorphisms and neither of the two spectra is spatial.
\end{prop}
\begin{proof}
The Nisnevich context $Int=(CRings^o=(IntClo^o,Int^o),\emptyset)$ is good: b. is true by proposition \ref{intgen} and distinguished families were defined in lemma \ref{fpintsurj}. We apply theorem \ref{pointspec}.
\end{proof}

The two notions of points agree.
\begin{prop}\label{sameptintsurj}
For $A\in CRings$, the set of points of $Spec_{Prop}(A)$ is in bijection with $pt_{Prop}(A)$.
\end{prop}
\begin{proof}
We need to prove that the set of points of $Spec_{Prop}(A)$ is in bijection with that of prime ideals of $A$.
We proceed as in proposition \ref{sameptet}.
Given a prime ideal and the associated integral domain quotient $A\to A/p$, we consider $A/p\to K(A/p)\to K(A/p)^{alg}$ where $K(A/p)^{alg}$ is an algebraic closure of the fraction field $K(A/p)$. The $(Int,IntClo)$ factorisation of this maps defines an object $(A/p)^{sint}$ which is a point of $Spec_{Prop}(A)$. $(A/p)^{sint}$ is called the {\em strict integral closure} of $A$ at $p$.
The map $A/p\to K(A/p)^{alg}$ is injective and so is $A/p\to (A/p)^{sint}$ which implies that $p$ is the kernel of $A\to (A/p)^{sint}$. We have constructed an injective map from prime ideals to points of $Spec_{Prop}(A)$. We prove now the surjectivity : for $A\to B$ a point of $Spec_{Prop}(A)$, $B$ being an integral domain, the kernel of $A\to B$ is a prime ideal ; with the notation of before, we have a diagram
$$\xymatrix{
A\ar[r]^-{Int}\ar[d]_{Surj}&B\ar[r]^-{IntClo}&K(B)\\
A/p\ar[r]_-{Int}\ar[ru]_{Mono}&(A/p)^{sint}\ar[r]_-{IntClo}&K(A/p)^{alg}\ar[u]_{IntClo}
}$$
presenting $A\to (A/p)'\to K(B)$ as another factorisation of $A\to B\to K(B)$, so $B\simeq (A/p)^{sint}$.
\end{proof}

\subsection{Nisnevich finite topology}\label{nisfinite}

Integral domains, integrally closed domains and strict integrally closed domains behave like local rings, henselian local rings and strictly henselian local rings, so it is tempting to define a Nisnevich localisation of the $(Int,IntClo)$ setting so that local object are non strict integrally closed domains.

We consider the class $\cal F$ of fields and the associated Nisnevich forcing of the previous setting.
A $(Int,IntClo)$ point covering family $\{A\to A_i,i\}$ of $A$ is $\cal F$-localising iff for any map $A\to K$ to a field, there exists an $i$ and a factorisation of $A\to K$ through $A\to A_i$. 
In particular, $(Mono^o,Surj^o)$ point covering families are $\cal F$-localising.
The Nisnevich context $NFin:=(CRings^o=(IntClo^o,Int^o),{\cal F})$ will be called the {\em Nisnevich finite context}.

\begin{prop}
A ring is in the saturation of $\cal F$ iff it is an integrally closed domain.
\end{prop}
\begin{proof}
Let $A$ be an integrally closed domain, \ie a integral domain such that the map $A\to K(A)$ to the fraction field is integrally closed, and $A\to A_i$ a $\cal F$-localising point covering family. By definition of such a family there exists an $i$ and a factorisation $A\to A_i\to K(A)$ of $A\to K(A)$. This forces $A\to A_i$ to be an integral extension and, as $A$ is integrally closed, there exists a retraction.
The reciprocal part has already been proven in the proof of proposition \ref{localintsurj}.
\end{proof}

The following lemma is a consequence of $Surj\subset Intsurj$ and of the definition of Nisnevich covering families.
\begin{lemma}
$(Mono^o,Surj^o)$ point covering families are Nisnevich finite covering families.
\end{lemma}

\subsubsection{Distinguished covering families}\label{fpcfnisdual}

A Nisnevich finite point covering family $\{B\to B_i,i\}$ of an $A$-algebra $B$ is said {\em distinguished} if it is of finite presentation over $A$, \ie there exist $A\to B'\to B$ where $A\to B'$ is of finite presentation and all $B\to B_i$ are pushout of some maps $B'\to B'_i$ between algebra of finite presentation, and satisfies one of the following two conditions
\begin{enumerate}[a.]
\item it is a $(Mono^o,Surj^o)$ point covering family, 
\item or the family is reduced to two elements $B'\to B'_0$ and $B'\to B'_1$ where $B'_0=B/b$ for some $b\in B$ and $B'_1$ is an integrally extension of $B'$ such that $B'[b^{-1}]\to B'_1[b^{-1}]$ admit a retraction.
\end{enumerate}
Geometrically (for the Zariski topology), this last condition says that the covering family is distinguished if it contains a finitely presented Zariski closed set $Z$ and cover its complement by some integral extension that has a section over the complement of $Z$. 

\begin{prop}\label{distnisfin}
A $A$-algebra $B$ is an integrally closed domain iff it lifts through any distinguished Nisnevich finite covering families.
\end{prop}
\begin{proof}
We need to prove only the sufficient part. 
Lifting through finitely presented $(Mono^o,Surj^o)$ point covering families says that $B$ is an integral domain (lemma \ref{proofintegral}), we need then to show that, if $K(B)$ is it fraction field of $B$, the map $B\to K(B)$ is integrally closed, \ie has the left lifting property with respect integral map between finitely presented $A$-algebras of the type $C\to (C/I)[x]/P(x)$ for some finitely presented $I$ and some monic polynomial $P$, we can use the same trick as in proposition \ref{fpintsurj} and suppose $I=0$. We are going to complete $C\to D$ into a distinguished Nisnevich finite covering family. In a diagram
$$\xymatrix{
C\ar[r]\ar[d]&B\ar[d]\\
C[x]/P(x)\ar[r]\ar[ru]^\ell&K(B)
}$$
we can always assume $C$ to be an integral domain by quotienting by the kernel of $C\to K(B)$, so $x$ can be describe in $K(C)$ as some fraction $a/b$ so $C[b^{-1}]\to C[b^{-1}][x]/P(x)$ has a section. This will be the distinguished localisation of the covering family, we complete it in a cover with $C\to C/b$. Now by hypothesis $C\to B$ will factor one of the two maps of the cover, and it cannot be $C\to C/b$ as the map $K(C)\to K(B)$ send $b$ to an invertible element.
\end{proof}

\subsubsection{Spectra and moduli interpretation}\label{dualnispec}

\begin{prop}
For a ring $A$, points of $SPEC_{NFin}(A)$ are $A$-algebras that are integrally closed domains
and points of $Spec_{NFin}(A)$ are integral extension of quotients of $A$ that are integrally closed domains.
\end{prop}
\begin{proof}
The Nisnevich context $NFin$ is good: b. is true by proposition \ref{intgen} and distinguished families were defined in lemma \ref{distnisfin}. We apply theorem \ref{pointspec}.
\end{proof}

As in \S\ref{nispec}, the small Nisnevich finite spectrum of $A$ have in general more points than the set of prime ideals of $A$. 
Also, a prime ideal $p$ of $A$ still define two points of $Spec_{NFin}(A)$, the first one is the point of $Spec_{Prop}(A)$ associated to $p$ and the second on is the integrally closed domain obtained by the $(Int,IntClo)$ factorisation of the residue map $A\to \kappa(p)$.

\subsection{Remarks on the previous settings}\label{duality}

\paragraph{Etale-Finite comparison}

We would like to sketch here a parallel between the six previous studied contexts. Recall that $\cal F$ is the subcategory of $CRings$ generated by fields, and that $f{\cal F}$ that generated by fat fields (\S\ref{zarpoints}). 

\begin{center}
\begin{tabular}{c|c|c}
& Etale context & Finite context\\
\hline
Primary factorisation system &$(indEt,Hens)$ & $(Int,IntClo)$\\
Secondary factorisation system &$(Loc,Cons)$ & $(Surj,Mono)$\\
Nisnevich context &$((Loc,Hens),f{\cal F})$ & $((Surj,Mono),{\cal F})$
\end{tabular}
\end{center}
The 'secondary factorisation system' is obtained from the primary one by looking only at those maps in the left class that are epimorphisms in $CRings$: localisations are those ind-etale maps that are epimorphisms and surjections are those integral maps that are epimorphisms. The secondary factorisation system can be thought as a way to extract open embbedings from etale maps.
Also both Nisnevich localising classes are exactly the points of the secondary factorisation system.
We are not sure how much these remarks are meaningful, but they do sketch a general structure. Thinking as $\cal C$ as $CRings^o$, one can define canonically from a given factorisation system ${\cal C=(A,B)}$ a secondary factorisation system as $(^\bot({\cal B}\cap Mono),{\cal B}\cap Mono)$ and a Nisnevich context $({\cal C=(A,B)},{\cal P}t_{({\cal B}\cap Mono)^f}({\cal C}))$.

\paragraph{Points and local objects}
We recall the comparison between the points and local objects for the different contexts.
\begin{center}
{\scriptsize
\begin{tabular}{c|c|c}
& Etale context & Finite context\\
\hline
Secondary points & fat fields & fields\\
Primary points & fat separably closed fields & algebraically closed fields\\
Secondary local objects & local rings & integral domain\\
Primary local objects & strict Henselian local rings & strict integrally closed domains\\
Nisnevich local objects & Henselian local rings & integrally closed domains
\end{tabular}
}
\end{center}

It is remarkable that for the four factorisation systems the set of points of a ring $A$ is always the set of prime ideals of $A$ and that it always coincide with the set of points of the associated spectra (but this is no longer the case for the associated Nisnevich contexts).
Also for every prime ideal $p\subset A$ there exists always a (essentially unique) distinguished map $A\to \kappa(p)^!$ where $\kappa(p)^!$ is the residue field or some extension of it at $p$, such that the local object at $p$ can be constructed by factorising $A\to \kappa(p)^!$ for the underlying factorisation system.

\paragraph{Duality}
More than a formal analogy, there seems to be a kind of duality relating the Etale and Finite contexts. We do not know how to formalize this duality, but it is intuitively related to the complementarity between localisation and quotients.
The following table points out other dual notions. 

\begin{center}
{\scriptsize
\begin{tabular}{c|c|c}
& Etale context & Finite context\\
\hline
secondary generators & $\ZZ[x]\to\ZZ[x,x^{-1}]$ & $\ZZ[x]\to\ZZ$ \\
& ($\GG_m\to\AA^1$)& ($\{0\}\to \AA^1$)\\
locality condition & $x+y$ invertible $\Rightarrow$ $x$ or $y$ invertible & $xy=0 \Rightarrow (x=0 \textrm{or }y=0)$\\
& ($x+y\in \GG_m \Rightarrow x$ or $y\in \GG_m$) & \\
completion &henselisation &normalisation\\
\end{tabular}
}
\end{center}

Normalisation of a noetherian ring $A$: if $p_i$ are the minimal prime of $A$ and $\kappa(p_i)$ the associated residue fields, $NA$ is the middle object of the $(Int,Intclo)$ factorisation of $A\to \prod_i \kappa(p_i)$. It is always a product of the normalisation $NA_i$ of the $A/p_i$, indeed the idempotents associated with $\prod_i \kappa(p_i)$ are elements integral over $\ZZ$ so they belong to $NA$.

Henselisation of a semilocal ring $A$: if $m_i$ are the maximal prime of $A$ and $\kappa(m_i)$ the associated residue fields, $HA$ is the middle object of the $(indEt,Hens)$ factorisation of $A\to \prod_i \kappa(m_i)$. As $A$ is the product of its localisations $A_{m_i}$, $HA$ is the product of the henselisation $HA_i$ of the $A_{m_i}$.

\paragraph{Dual lifting properties} This would-be-duality between the Etale and Finite contexts can be also thought as follow.
Having in mind that points of a local ring are all generisation of the closed point, and that points of a integral domain are all specialisation of the generic point, the lifting properties for etale and finite maps are dual in the same sense than a category and its opposite.
Another illustration of this is the fact that the poset of points of Zariski and Domain spectra are opposite categories.
All this recall Grothendieck's smooth and proper functors \cite{georges1} for which a functor $F:C\to D$ is smooth iff its opposite $F:C^o\to D^o$ is proper. It is stated in \cite{georges1} that this property of functors has no analog in algebraic geometry, but these dual topologies could be a hint toward a more precise analogy. However the classes of smooth and proper functors are not know (yet?) to be part of factorisation systems so a link with our theory is not obvious.

\paragraph{Toward a formalisation}
The example to follow in \S\ref{leftright} of left and right fibrations of category also has a flavour of the same kind of duality, but the situation is clearer in this setting as the opposition of categories exchange the two dual factorisation systems. Is there an operation of the same kind exchanging the Etale and Finite factorisation systems ?
Also, Pisani in \cite{pisani} proposes a formalization of dual factorisation systems through the property that he call the {\em reciprocal stability law} (see our \S\ref{leftright}); is there such a structure between Etale and Finite systems ? do they define a Pisani duality ?

\subsection{Other examples}\label{other}

In this section, we would like to work with our general idea of topologizing factorisation systems in categories more general than opposite of locally finitely presentable ones. The reason for such a generalization is that the definition of points and of spectra give back known objects. The main difference between the examples to follow and our general setting of \S\ref{topo} is the absence of the class of finitely presented maps, so we have generalized our definitions by using all maps instead of finitely presented ones.

We sketch only the results of the study, proofs are left to the reader.

\subsubsection{$(Loc,Cons)$ topology in the category of monoids with zeros}

Let $({\cal S}_*,\wedge)$ be the monoidal category of pointed sets with the smash product, monoids in this category are called {\em monoids with zero} and their category is noted $Mon_*$. If $Mon$ is the category of monoids in $({\cal S},\times)$, there is an obvious functor $Mon\to Mon_*$ sending a monoid $M$ to $M_*$ which is $M$ with an extra absorbing element.

In the same way as $CRings$, the category $Mon_*$ can be equiped with a $(Loc,Cons)$ factorisation system. 
Localisations are the natural notion and conservative maps are defined as having the right lifting property with respect to the natural inclusion $\NN_*\to\ZZ_*$ (which is a left generator for the whole system).

This system is studied in details in \cite{anelvaquie}, where the factorisation topology is proved to coincide with the Zariski topology of \cite{fun}. As a consequence, our factorisation spectrum do coincide with the one defined in \cite{fun}.

\subsubsection{$(Epi,Mono)$ topology in a topos}

We investigate the (Epi,Mono) factorisation system of maps of a topos $\cal T$.

An object $P\not=\emptyset\in\cal T$ is a point iff any monomorphism $U\to P$ ($U\not=\emptyset$) admits a section. 
This forces $U\to P$ to be an isomorphism : points are objects without any proper subobject. These objects are called {\em atoms} of the topos \cite[C.3.5.7]{elephant}.
Maps between atoms are always epimorphisms and all quotients of atoms are atoms.
Points of an object $X$ are called atomic subobjects of $X$, any two atomic subobjects are either equal or disjoint in $X$.
Any morphism $A\to X$ with $A$ an atom factors through a unique atomic subobject of $X$, so the set of points of $X$ is that of its atomic subobjects.
The family of all atomic subobjects of $X$ is the finest point covering of $X$, so local objects coincide with points and $Spec_{atom}(X)$ is the topos of presheaves over the set of atomic subobjets of $X$.

We are going to illustrate this in the topos $BG=G\z Sets$ classifying $G$-torsors for some discrete group $G$. Objects of $BG$ are sets with a right action of $G$ and can be thought as particular groupoids, a map is a monomorphism if, viewed as a map of groupoids, it is fully faithful.
Points of $(Epi,Mono)$ system of $BG$ are sets with a transitive action of $G$. The category of all points is then the orbit category of $G$ and the set of points of $X\in BG$ is simply the set of orbits of the action of $G$. A point covering family is a family of monomorphisms surjective on orbits, or view through the associated groupoids, a family of fully faithfull maps globally essentially surjective. 
The family of all orbits of a given $X$ is the finest point covering family of $X$, and $Spec_{atom}(X)$ is equivalent to the topos of presheaves on the set of orbits of $X$.

\subsubsection{$(Epi,Mono)$ topology in an abelian category and discrete projective spaces}\label{abel}

Any abelian category $\cal C$ has an $(Epi,Mono)$ unique factorisation system, its initial object $0$ is also final and so not strict but this is not important.

Points are non zero objects without any proper subobject, \ie simple objects. 
Any map to $M$ from a simple object is either 0 or a monomorphism, the set of points of $M$ is then the set of simple or null subobjects of $M$.
The family of all simple subobject of $M$ is the finest point covering family of $M$, so all local objects are points and the small spectrum $Spec_{Epi}(M)$ is the topos of presheaves on the poset of simple or null subobjects of $M$. All simple subobjects correspond to closed points and 0 to a generic point.

If ${\cal C}$ is the category of vector spaces over some field $k$, $Spec_{Mono}(M)$ is a sort of discrete projective space for $M$, with an extra generic point. Forgetting about this generic point, a map $M\to N$ can be though as inducing a partially defined transformation (it is not defined on the kernel of $M\to N$) between the associated projective spaces.

The big spectrum $SPEC_{Mono}(0)$ is the category of presheaves over the category of simple objects of $\cal C$.
And the structure sheaf map $Spec(M)\to SPEC(0)$ send a simple subobject of $M$ to its underlying simple object.

\medskip
The dual system $({\cal C}^o=(Epi,Mono)=(Mono^o,Epi^o)$ is also interesting.
Points of ${\cal C}^o$ are objects without any proper quotient, which are again simple objects; the set of points of an object $M$ is that of simple or null quotients of $M$ and $Spec_{Epi}(M)$ is a the "dual" projective space of $Spec_{Mono}(M)$ still with an extra point, which is this time the only to be closed.

\subsubsection{Discrete fibrations of categories}\label{leftright}

We are going to study two unique factorisation systems on the category $CAT$ of small categories, the references for all the results are \cite{joyal,pisani}.

Let $[n]$ be the ordinal with $n+1$ elements $0<\dots<n$ viewed as a category. $[0]$ is the punctual category. The two functors $[0]\to [1]$ will be called $\underline{0}$ and $\underline{1}$.
In $CAT$, the unique factorisation system $(Fin,DRFib)$ is defined as left generated by $\underline{1}$, $Fin$ is called the class of {\em final functors}, and $DRFib$ the class of {\em discrete right fibrations}.
There is a dual system $(Ini,DLFib)$ left generated by $\underline{0}$, $Ini$ is called the class of {\em initial functors}, and $DLFib$ the class of {\em discrete left fibrations}. It is easy to see that $C\to D\in LFib$ iff $C^o\to D^o\in RFib$.

We are only going to detail the factorisations in a special case: if $c:[0]\to C$ is an object of a category $C$, the $(Fin,DRFib)$ factorisation of $c$ is $[0]\to C{/c}\to C$ and the $(Ini,DLFib)$ factorisation of $c$ is $[0]\to c/C\to C$. 
We want say much of the left classes only that in the previous factorisation $[0]\to C{/c}$ points the final object of $C{/c}$ and $[0]\to c/C$ the initial object of $c/C$.
As for the right classes, it can be shown that any $D\to C\in DRFib$ is associated a presheaf $F:C^o\to {\cal S}$ such that $D$ is isomorphic to $C{/F}$ and that any $D\to C\in DRFib$ is associated a functor $F:C\to {\cal S}$ such that $D$ is isomorphic to $F/C$.
From this we can deduced that the categories $DRFib{/C}$ and $DLFib{/C}$ are respectfully equivalent to the category $\widehat{C}$ of contravariant functors $C\to {\cal S}$ and to that $\check{C}$ of covariant functors $C\to{\cal S}$.

We are now going to study the $(Fin,DRFib)$ system, the associated factorisation topology will be called the {\em right topology}. 
A point is a non empty category $P$ such that any any discrete right fibration $C{/F}\to C$ has a section. Using the Yoneda embedding in $\widehat{C}$, this condition says every presheaf on $C$ has a global section. Such categories can be highly non trivial ($\Delta$ is an example) and the set of points of category is difficult to described, but fortunately the point covering families are simple to understand.
Certainly $[0]$ is a point, and so a point covering family of $C$ has to be globally surjective on the objects of $C$. This condition is also sufficient: indeed if $P\to C$ is a point of $C$, it will lift through a covering family $U_i\to C$ iff one of the fiber product $U_i\times_CP$ is not empty, but if $U_i\to C$ is assumed surjective on the points, it cannot happen that all fiber products are empty.

A local object is a category such that any epimorphic family of presheaves contains a presheaf with a global section.
In particular any category with a terminal object is a local object (as proven already by the factorisation $c:[0]\to C{/c}\to C$).
We don't know if all local object are of this type, neither if they are all pointed.

A discrete right fibration $C{/F}\to C$ is surjective on the objects iff $F(c)\not=\emptyset$ for all $c\in C$ iff $F\to \in\widehat{C}$ is an epimorphism. In the same way a family $U_i\to C$ of discrete right fibrations is globally surjective on objects iff it is globally epimorphic in $\widehat{C}$. The small site of $C$ is $\widehat{C}$ and the previous remark show that the topology is the canonical one, so $Spec_{Right}(C)$ is the topos $\widehat{C}$. Its category of points is that of pro-objects of $C$. 

$SPEC_{Right}(C)$ is the topos of presheaves over $CAT{/C}$. Every object $c\in C$ define a point of $\widehat{C}$, the associated local object is $C{/c}$ and the structural map is $C{/c}\to C$. Using a topological vocabulary, one can say that $C{/c}$ is the {\em right localisation} of $C$ at $c$.

For the $(Ini,DLFib)$ system the same reasonning leads a topology called the {\em left topology} and to $Spec_{Left}(C)$ being the topos $\check{C}$.

\paragraph{Analogy with the Etale-Finite duality}
The pair of $(Fin,LFib)$ and $(Ini,RFib)$ looks dual in the same sense that $(Loc,Cons)$ and $(Surj,Mono)$ or $(IndEt,Hens)$ and $(Int,IntClo)$ are in $CRings$.
$(Fin,LFib)$ is left generated by $1:[0]\to [1]$ and $(Ini,RFib)$ is left generated by $0:[0]\to [1]$, thinking of $[1]=0\to 1$ as a specialisation morphism, $0$ is then generic point and $1$ the closed point. With this vocabulary a discrete right fibration lift any generisation of any object that is lifted and so behave as an open map, and a discrete left fibration lift any specialisation of any object that is lifted and so behave like a closed map.
This situation is to compare with the facts that Zariski open embeddings lift any generisation of any point that is lifted and that closed embeddings lift any specialisation of any point that is lifted.

Also, the generators $\GG_m\to \AA^1$ and $\{0\}\to \AA^1$ of the $(Loc,Cons)$ and $(Surj,Mono)$ systems on $CRings$, which also are a generic point and a closed point. However, seen geometrically in $CRings^o$ the generators are this time in the right class. 

\medskip
Moreover in this case, $\widehat{C}$ and $\check{C}$ have a duality pairing given by the coend:
\begin{eqnarray*}
\widehat{C}\times\check{C}&\tto &{\cal S}\\
(F,G)&\longmapsto &\int^{C}F\times G
\end{eqnarray*}
This pairing is moreover "exact" in the sense that the natural map $\widehat{C}\to CAT(\check{C},{\cal S})$ is an equivalence on the subcategory of functors commuting with all limits and $\check{C}\to CAT(\widehat{C},{\cal S})$ is an equivalence on the subcategory of functors commuting with all colimits.

Is this a feature of the same duality ? Does a similar pairing exist for spectra of rings or Pisani dual systems \cite{pisani} ?

\paragraph{Locality properties between the two systems}
Those two system have also some compatibility conditions together, call a reciprocal stability law between two factorisation systems by C. Pisani in \cite{pisani}. The left class of a factorisation system is not in general stable by base change but $Fin$ and $Ini$ are stable by base change along $DLFib$ and $DRFib$ respectively.
This has an interesting consequence as a map $C\to D$ can be characterized to be final iff its pull-back along every $d/D\to D$ for some $d\in D$ is final 
$$\xymatrix{
d/C\ar[r]\ar[d]&C\ar[d]\\
d/D\ar[r]& D.
}$$
Now this can be read using a topological langage: $d/C$ is the localisation of $C$ at $d$ in $\widehat{D}$ and being a final maps is a local property for the Right topology. Dually of course, being initial is a local property for the Left topology.
Also, these topologies can be used to interpret Quillen's theorem A and many definitions of \cite{georges1} as proving locality properties of some classes of functors with respect to the left or right topology.


\paragraph{Groupoids}
Restricted to the category of groupoids, $DRFib$ and $DLFib$ coincide and define the class of {\em coverings functors} (discrete fibrations)
and $In$ and $Fin$ coincide too and define the class of {\em connected functors}. In fact both factorisation systems restrict to the categrory of groupoids and define a factorisation system compatible with weak equivalence such that, when groupoids are taken as models for homotopy 1-types, it induces the 0-th Postnikov system of \S\ref{highercase}.

\subsubsection{A dual topological realisation for simplicial sets}

This is a funny application of our notion of Nisnevich context.

Let $\Delta$ be the category of finite (non empty) ordinals and order preserving maps.
Writing $[n]:=(0<1<\dots <n)$ for the $(n+1)^{th}$ ordinal, a map $u:[n]\to [m]\in\Delta$ decomposes into $[n]\to [p]\to [m]$ where $[n]\to [p]$ is a {\em surjection} and $[p]\to [m]$ a {\em monomorphism}. This factorisation system is left generated by the single map $[1]\to[0]$.

The category $SSets=\widehat\Delta$ of presheaves on $\Delta$ is the category of simplicial sets, objects of $\Delta$ view in $SSet$ will be noted $\Delta[n]$ and called simplices. 

\begin{lemma}\label{extensionfacto}
If ${\cal C}$ is a full subcategory of a cocomplete category ${\cal D}$, any unique factorisation system ${\cal C}=({\cal A},{\cal B})$ left generated by compact objects extend to a unique unique factorisation system $\cal D=({\cal A}',{\cal B}')$ such that $\cal A=A'\cap C$ and $\cal A=A'\cap C$.
\end{lemma}
\begin{proof}
Let $G$ be a set of left generators, so ${\cal B}=G^\bot$ and ${\cal A}=^\bot{\cal B}$ in $\cal C$.
We define now ${\cal B}':=G^\bot$ and ${\cal A}':=^\bot{\cal B}'$ in $\cal D$.
It is clear that ${\cal C}\cap {\cal B}'={\cal B}$ and so we have also ${\cal C}\cap {\cal A}'={\cal A}$.
Now the set of generators $G$ can always be completed to satisfies assumptions of proposition \ref{factolim} so we only have to prove that the factorisation in $\cal D$ of a map in $\cal C$ coincide with the factorisation in $\cal C$, but this is obvious by unicity of the factorisation.
\end{proof}

\begin{cor}
The unique factorisation system $(Surj,Mono)$ on $\Delta$ can be extended to $SSet$ in a system noted $(Deg,NDeg)$.
\end{cor}
A map in $Deg$ will be called {\em degenerated} and a map in $NDeg$ {\em non degenerated}.

\begin{prop}
$NDeg$ is the class of maps of simplicial sets $u:Y\to X$ sending non degenerate simplices of $Y$ to non degenerate simplices of $X$.
In particular, a map $\Delta[n]\to X$ is in $NDeg$ iff it is a non degenerate simplex of $X$.
\end{prop}
\begin{proof}
First we claim that a particular case of the factorisation is the one of the Eilenberg-Zilber lemma saying that a map $\Delta[n]\to X\in SSet$ factors through a unique $\Delta[n']$ where $n'\leq n$ so that the map $\Delta[n]\to \Delta[n']$ is a surjection and $\Delta[n']\to X$ is a non degenerate simplex. So the simplex $\Delta[n]\to X$ is degenerated iff $n'<n$. Using this factorisation on the top and bottom arrows, we can develop any lifting square in
$$\xymatrix{
\Delta[n]\ar[r]^-{\textrm{surj.}}\ar[d]_{\textrm{surj.}}&\Delta[n']\ar[r]^-{\textrm{non deg.}}\ar[d]&Y\ar[d]\\
\Delta[m]\ar[r]_-{\textrm{surj.}}&\Delta[m']\ar[r]_-{\textrm{non deg.}}&X
}$$
where $\Delta[n']\to \Delta[m']$ is a surjection by cancellation.
The map $Y\to X$ is orthogonal to surjection of simplices iff the map $\Delta[n']\to \Delta[m']$ is an isomorphism. But this condition says exactly that a non degenerated simplex of $Y$ is send to a non degenerated simplex of $X$.
\end{proof}

\paragraph{Raw spectrum}

\begin{prop}
The only point is $\Delta[0]$.
\end{prop}
\begin{proof}
It is easy to see that $\Delta[0]$ is a point.
Conversely, a simplicial set $X$ is a point if $Y\to X\in NDeg$ every it admit a section. Applied to $\Delta[0]\to X$ this forces $X$ to be $\Delta[0]$.
\end{proof}

The set of points of an object $X$ is exactly the set of vertices $X$.
A family of maps $U_i\to X\in NDeg$ is a point covering family iff it is surjective on vertices. 
For any simplicial set $X$, the family of maps $\Delta[0]\to X$ is the finest cover of $X$.
As a consequence, the only local simplex is $\Delta[0]$ (and of course every local object is pointed local).

\begin{prop}
$Spec_{NDeg}(X)\simeq {\cal S}^{X_0}$.
\end{prop}
\begin{proof}
For any $U\to X$, the nerve of the covering by simplices of $U$ is constant si a presheaf $F:NDeg{/X}^o\tto {\cal S}$ is a sheaf for the factorisation topology iff $F(u:U\to X) = \prod_{x\in U_0} F(u(x))$.
\end{proof}

\paragraph{Simplectic Nisnevich Spectrum}
To make this setting a bit more interesting, we are going to make a Nisnevich localisation along the category $\Delta$ of simplices.
Covering families of the Nisnevich context $\Delta Nis:=(SSets=((Deg,NDeg),SSets,\Delta)$ are families of maps $U_i\to X\in NDeg$ lifting not only vertices but any simplex of $X$.

\begin{lemma}\label{deltaniscover}
The family of all maps $\Delta[n]\to X\in NDeg$ for all $n$, is a Nisnevich covering family of $X$.
\end{lemma}
\begin{proof}
We need to prove that any $\Delta[m]\to X$ factors through one of the $\Delta[n]\to X\in NDeg$, but this Eilenberg-Zilber lemma.
\end{proof}

\begin{cor}
Local objects of the Nisnevich context $\Delta Nis$ are simplices.
\end{cor}
\begin{proof}
By definition of the context, simplices are local.
Conversely by lemma \ref{deltaniscover} it is enough to use the family of all $\Delta[n]\to X\in NDeg$. Let $d:\Delta[n]\to X$ be a map of the family having a section $s$, $s$ is in $NDeg$ and so is $sd$. But the only non degenerate endomorphism of $\delta[n]$ is the identity, so $d$ is an isomorphism.
\end{proof}

As a consequence, the set of points of the Nisnevich spectrum $Spec_{\Delta,NDeg}(X)$ is the set of maps $\Delta[n]\to X\in NDeg$, \ie the set of non degenerate simplices of $X$.

\begin{prop}
Let $P(n)$ be the poset of faces of $\Delta[n]$.
$Spec_{\Delta Nis}(\Delta[n])$ is the topos of presheaves over $P(n)$.
In particular this is a spatial topos whose poset of points is $P(n)$.
\end{prop}
\begin{proof}
For the first assertion, we just need to prove that the topology is trivial, but any cover of $\Delta[m]$ admits a copy of $\Delta[m]$ so the identity is the finest cover.
The category of points is $Pro(P(n))$ which turns out to be equivalent to $P(n)$. This is a consequence of the fact that any functor $f:I\to P(n)$ where $I$ is a filtered category factors through a category $J$ with a terminal object (hence every pro-object will be representable). To see this it is enough to consider $I$ to be a poset, and a poset is filtered iff for any two objects $i$ and $j$, there exists an object $k$ and two arrows $k\to i$ and $k\to j$. If $f:I\to P(n)$ is a filtered diagram, $f(i),f(j)$ and $f(k)$ are faces of $\Delta[n]$ and if $f(i)$ is a vertex then necessarily $f(k)=f(i)$ and $f(i)$ is a vertex of $f(j)$. This implies that there can be at most one vertex of $\Delta[n]$ in the image of $f$ and this vertex is a terminal element for the image poset of $f$, proving our assertion. If no vertices are in the image of $f$, there can be at most a single edge in the image of $f$ which is then the terminal element of the image poset. If no edges are in the image of $f$, one has to continue the same argument with higher dimensional faces.
\end{proof}

\begin{cor}
$Spec_{\Delta Nis}(X)$ is topological space such that any non-degenerate $\Delta[n]\to X$ is an open embedding.
\end{cor}

The small Nisnevich spectra of a simplicial set $X$ can be thought as a geometric realisation of $X$ as it is a spatial object that does not see the degenerate part of $X$. This geometric realisation is such that any vertex of $X$ is open in $Spec_{\Delta Nis}(X)$ and as show the computation of $Spec_{\Delta Nis}(\Delta[n])$, it can be thought as a cellular complex dual of the usual geometric realisation (use for example in the theory of Poincar\'e duality).

This "duality" raises the question of the existence of another factorisation system on $\widehat{\Delta}$ for which the small spectra of a simplicial set would be (a combinatorial form of) the usual geometric realisation. Unfortunately, for this realisation, the only open of a $n$-simplex would be the cell of dimension $n$ but such a cell without its boundary is not a simplicial object. In fact, ordinary geometric realisation being constructed by glueing along closed subsets, they are not local for the topology of the realisation and topossic techniques do not seem relevant here.

\appendix

\section{Comparison with other works}

Since \cite{SGA4-1}, a lot of works have been done on the problem of building a general theory of spectra, for example \cite{hakim,johnstone,dubuc,lurie}. 

In \cite{johnstone}, Johnstone uses only a partial factorisation system (in the example of rings, this would be a factorisation only for maps from a ring to a local ring). His axiomatization produces a more general theory than ours but at the price of being more complex, I think the consideration of a full factorisation system is more natural.

In \cite{dubuc}, Dubuc considers a topos $\cal T$ and its category of points $P$, he uses a class $\cal C$ of maps in $\cal P$ to define his {\em etal class} $\cal E$ of maps of $\cal T$ by a property of right orthogonality then construct the spectrum of a sheaf $F$ as ${\cal E}/F$.
In the case of a presheaf topos $T=\widehat{D}$ endowed with a factorisation system $\cal (A,B)$ on its category of points $Pro(D)$ (as in our examples in algebraic geometry), we claim that the Etal class $\cal E$ orthogonal to $\cal A$ is such that $\cal B=E\cap D$ and that his and our notion of spectra agree (although his is defined for any sheaf, not only for objects of the site, but our definition can be extended easily).

As it uses any topos and not only presheaves ones, this setting is more general than our, particularly factorisation systems are not considered although lifting systems are implicit.

In \cite{lurie}, Lurie defines what he call a {\em geometry} which is essentially a small category $D$ with finite limits, a subcategory $D^{ad}$ satisfying some axioms and a Grothendieck topology on $D^{ad}$. This compares well enough to our setting: with his axioms, ${\cal C}=Pro(D)$ is the opposite of a lcoally finitely presentable category and $D^{ad}$ generates on the left a factorisation $\cal (A,B)$ on $\cal C$ such that ${\cal B}\cap D=D^{ad}$. Then he considers arbitrary topologies that can been defined by families of maps in $\cal B$ and this is were our works add something: having at our disposal the notion of point of a factorisation system, we can define a distinguished topology.

Also Lurie works in the more general setting of higher categories, but we claim our constructions work exactly the same in this setting, provided unique factorisation systems are replaced with homotopically unique factorisation system.

\bigskip
In all those three settings, factorisation systems play a side role whereas they are our main object of consideration. 
Each setting has its own advantages and that of our is in the definition of the notion of points of a factorisation system and in their use to define Grothendieck topologies such as Zariski or Etale topologies and others.

\end{document}